%% file: MaPaper.tex
\documentclass[a4paper,10pt]{article}

\usepackage[utf8]{inputenc} 
\usepackage{color}
\usepackage[ruled, vlined]{algorithm2e}
\usepackage{relsize}

\input{preamble} 
\input{math} 

\graphicspath{{images/}} 

\newcommand{\sk}[1]{{\color{black}{#1}}}

\title{\bfseries Approximation of Generalized Ridge Functions in High Dimensions}
\author{\hspace*{-1cm}
  Sandra Keiper\\[.5em]
 \textsc{\hspace*{-1cm}Department of Mathematics, Technische Universit\"at Berlin}\\
  10623 Berlin, Germany \\[.5em]
  \hspace*{-1cm}E-mails: \href{mailto:keiper@math.tu-berlin.de}{keiper@math.tu-berlin.de}
}

\begin{document}
\listoftodos
\input{content-macros}
\maketitle
\input{abstract}

\vspace{.1in}


\vspace{.1in}


\input{contents2016}

\vspace{.1in}
\section*{Acknowledgements.}
The author acknowledges support by the DFG Grant 1446/18 and the Berlin Mathematical School. In particular the author acknowledges Ingrid Daubechies, Gitta Kutyniok and Mauro Maggioni for helpful discussions. 

\vspace{.1in}
\nocite{*}
\bibliographystyle{amsplain}
\bibliography{MaPaper.bib}

\end{document}

%% file: preamble.tex
\makeatletter
\newcommand{\LoadPackagesNow}{}
\newcommand{\LoadPackageLater}[1]{%
   \g@addto@macro{\LoadPackagesNow}{%
      \usepackage{#1}%
   }%
}
\makeatother

\usepackage{xspace}
\usepackage{xargs}
\usepackage{ifthen}
\usepackage{etoolbox}

\usepackage[
	table 
]{xcolor}
\usepackage[%
]{graphicx}
\usepackage{float}
\usepackage{wrapfig}

\usepackage{subfigure}
\usepackage{caption}
\usepackage{multicol}
\usepackage[
   tbtags,    
   sumlimits,  
   nointlimits, 
   namelimits, 
   reqno,     
]{amsmath} %

\usepackage{amsfonts}
\usepackage{mathrsfs} 
\usepackage{dsfont}
\usepackage{amssymb}
\LoadPackageLater{amsthm}
\LoadPackageLater{thmtools}
\usepackage[fixamsmath,disallowspaces]{mathtools}
\mathtoolsset{showonlyrefs}
\mathtoolsset{centercolon=true}

\usepackage{geometry}
\usepackage[%
	square,	
	comma,	
	numbers,	
	sort,		
	sort&compress,    
]{natbib}
\let\chapter\section
\usepackage{algorithm2e}
\usepackage{framed}
\usepackage[normalem]{ulem}      
\usepackage{soul}		            
\usepackage{url} 
\usepackage{lipsum}
\usepackage[textsize=tiny,english,colorinlistoftodos,disable]{todonotes}

\usepackage{enumitem}
\definecolor{pdfurlcolor}{rgb}{0,0,0.6}
\definecolor{pdffilecolor}{rgb}{0.7,0,0}
\definecolor{pdflinkcolor}{rgb}{0,0,0.6}
\definecolor{pdfcitecolor}{rgb}{0,0,0.6}
\usepackage[
  colorlinks=true,         
  urlcolor=pdfurlcolor,    
  filecolor=pdffilecolor,  
  linkcolor=pdflinkcolor,  
  citecolor=pdfcitecolor,  %
  raiselinks=true,			 
  breaklinks,              
  verbose,
  hyperindex=true,         
  linktocpage=true,        
  hyperfootnotes=false,     
  bookmarks=true,          
  bookmarksopenlevel=1,    
  bookmarksopen=true,      
  bookmarksnumbered=true,  
  bookmarkstype=toc,       
  plainpages=false,        
  pageanchor=true,         
  pdftitle={Analysis of Generalized Ridge Functions},             
  pdfauthor={Sandra Keiper},            
  pdfcreator={LaTeX, hyperref, KOMA-Script}, 
  pdfdisplaydoctitle=true, 
  pdfstartview=FitH,       
  pdfpagemode=UseOutlines, 
  pdfpagelabels=true,           
  pdfpagelayout=SinglePage, 
]{hyperref}

\LoadPackagesNow

\usepackage{bm}

\makeatletter
\g@addto@macro\bfseries{\boldmath}
\makeatother

\newcommand{\ifargdef}[3][{}]{\ifthenelse{\equal{#2}{}}{#1}{#3}}



%% file: math.tex
\newtheoremstyle{claim}
	{\topsep}{\topsep}%
	{\itshape}
	{}
	{\bfseries\boldmath}
	{}
	{.5em}
	{\thmname{#1} \thmnumber{#2} \thmnote{(#3)}}
\newtheoremstyle{definition}
	{\topsep}{\topsep}%
	{}
	{}
	{}
	{}
	{.5em}
	{{\bfseries\thmname{#1} \thmnumber{#2}} \thmnote{(#3)}}
\newtheoremstyle{remark}
	{\topsep}{\topsep}%
	{}
	{}
	{}
	{}
	{.5em}
	{{\itshape\thmname{#1}:}}

\declaretheorem[style=claim,numberwithin=section]{theorem}

\declaretheorem[style=claim,sibling=theorem]{lemma}
\declaretheorem[style=claim,sibling=theorem]{corollary}

\declaretheorem[style=remark]{remark}
\declaretheorem[style=definition,sibling=theorem]{definition}

\newcommand{\opleft}[1]{\mathopen{}\left#1}
\newcommand{\opright}[1]{\right#1\mathclose{}}
\newcommandx{\braces}[4]{%
\ifstrequal{#3}{normal}{#1#4#2}{%
\ifstrequal{#3}{auto}{\left#1#4\right#2}{%
\ifstrequal{#3}{opauto}{\opleft#1#4\opright#2}{%
#3#1#4#3#2}}}%
}
\newcommandx{\opannot}[3][3=\downarrow]{\stackrel{\mathclap{\substack{#1 \\ #3 \vspace{2pt}}}}{#2}}
\newcommandx{\lineannot}[3][3=\rightarrow]{\mathllap{\boxed{\text{\textsmaller{#1}}} #3} #2}
\newcommandx{\multilineannot}[4][4=\rightarrow]{\mathllap{\boxed{\parbox{#1}{\RaggedRight\textsmaller{#2}}} #4} #3}
 
\newcommand{\N}{\mathbb{N}} 
\newcommand{\R}{\mathbb{R}} 

\newcommand{\suchthat}[1][normal]{\ifstrequal{#1}{normal}{\mid}{#1|}} 

\newcommand{\dist}[2]{\operatorname{dist}(#1, #2)} 
\newcommandx{\intvcl}[3][1=normal]{\braces{[}{]}{#1}{#2, #3}} 
\newcommandx{\intvop}[3][1=normal]{\braces{(}{)}{#1}{#2, #3}} 
\newcommandx{\intvclop}[3][1=normal]{\braces{[}{)}{#1}{#2, #3}} 
\newcommandx{\intvopcl}[3][1=normal]{\braces{(}{]}{#1}{#2, #3}} 

\DeclareMathOperator*{\argmin}{argmin} 
\DeclareMathOperator*{\sign}{sign} 
\newcommandx{\abs}[2][1=normal]{\braces{\lvert}{\rvert}{#1}{#2}} 
\newcommandx{\ceil}[2][1=normal]{\braces{\lceil}{\rceil}{#1}{#2}} 
\newcommandx{\floor}[2][1=normal]{\braces{\lfloor}{\rfloor}{#1}{#2}} 
\newcommandx{\round}[2][1=normal]{\braces{[}{]}{#1}{#2}} 
\newcommandx{\der}[1]{D^{#1}} 
\newcommandx{\partder}[4][1={},4={}]{\frac{\partial^{#4} #2}{\partial #3^{#4}}\ifargdef{#1}{\Big|_{#1}}} 
\newcommandx{\integ}[4][1={},2={}]{\int_{#1}^{#2} #3 \, #4} 
\newcommandx{\asympffaster}[2][1=normal]{o\braces{(}{)}{#1}{#2}} 
\newcommandx{\asympfaster}[2][1=normal]{O\braces{(}{)}{#1}{#2}} 
\newcommandx{\asympeq}[2][1=normal]{\Theta\braces{(}{)}{#1}{#2}} 
\newcommandx{\asympsslower}[2][1=normal]{\omega\braces{(}{)}{#1}{#2}} 
\newcommandx{\asympslower}[2][1=normal]{\Omega\braces{(}{)}{#1}{#2}} 

\newcommandx{\norm}[2][1=normal]{\braces{\|}{\|}{#1}{#2}} 
\renewcommandx{\sp}[3][1=normal]{\braces{\langle}{\rangle}{#1}{#2, #3}} 
\newcommand{\adj}[1]{{#1}^\ast} 
\newcommandx{\End}[2][2={}]{\mathcal{L}\opleft( #1 \ifargdef{#2}{, #2} \opright)} 
\DeclareMathOperator{\spann}{\operatorname{span}} 

\newcommandx{\measure}[2][1=normal]{\operatorname{vol}\braces{(}{)}{#1}{#2}} 
\newcommandx{\Leb}[3][1={},3=normal]{L^{#2}\ifargdef{#1}{\braces{(}{)}{#3}{#1}}{}} 
\newcommandx{\Lebnorm}[4][1=normal,3={2},4={}]{\norm[#1]{#2}_{\Leb[#4]{#3}}} 
\renewcommandx{\l}[3][1={},3=normal]{\ell^{#2}\ifargdef{#1}{\braces{(}{)}{#3}{#1}}} 
\newcommandx{\lnorm}[4][1=normal,3={2},4={}]{\norm[#1]{#2}_{\l[#4]{#3}}} 
\newcommandx{\Smooth}[4][1={},3={},4=normal]{C_{#3}^{#2}\ifargdef{#1}{\braces{(}{)}{#4}{#1}}} 
\newcommandx{\Schwartz}[2][1={},2=normal]{\mathscr{S}\ifargdef{#1}{\braces{(}{)}{#2}{#1}}} 
\newcommandx{\Schwartzpoly}[2][1=normal]{\braces{\langle}{\rangle}{#1}{\abs[#1]{#2}} } 
\newcommandx{\Tempdistr}[2][1={},2=normal]{\mathscr{S}'\ifargdef{#1}{\braces{(}{)}{#2}{#1}}} 
\newcommandx{\distrinp}[3][1=normal]{\braces{\langle}{\rangle}{#1}{#2, #3}} 
\newcommand{\Linedistr}[1][]{\mathfrak{L}\ifargdef{#1}{_{#1}}{}} 
\newcommandx{\ft}[3][1=default,2=auto]{
\ifstrequal{#1}{default}{\widehat{#3}}{
\ifstrequal{#1}{long}{{\braces{(}{)}{#2}{#3}}^{\wedge}}{}}} 
\newcommandx{\ift}[3][1=default,2=auto]{
\ifstrequal{#1}{default}{\check{#3}}{
\ifstrequal{#1}{long}{{\braces{(}{)}{#2}{#3}}^{\vee}}{}}} 


%% file: content-macros.tex
\newcommand{\OpAnalysis}[1]{T_{#1}} 
\newcommand{\OpSynthesis}[1]{\adj T_{#1}} 
\newcommand{\OpFrame}[1]{S_{#1}} 
\newcommand{\defsf}{\varphi} 
\newcommand{\InpSp}{\mathcal{H}} 
\newcommand{\InpSpK}{\InpSp_K} 
\newcommand{\ProK}{P_K} 
\newcommand{\InpSpM}{\InpSp_M} 
\newcommand{\ProM}{P_M} 
\newcommand{\sig}{x^0} 
\newcommand{\sigrec}{x^\star} 
\newcommand{\PF}{\Phi} 
\newcommand{\pf}{\phi} 
\newcommand{\cluster}{\Lambda} 
\newcommand{\concentr}[2]{\kappa\ifargdef{#1}{\opleft( #1, #2 \opright)}} 
\newcommand{\clustercoh}[2]{\mu_c \ifargdef{#1}{( #1 , #2)}} 
\newcommand{\anorm}[2]{\norm{#1}_{1,#2}} 
\newcommand{\ver}{\mathrm{v}} 
\newcommand{\hor}{\mathrm{h}} 
\newcommand{\dir}{\imath} 
\newcommand{\meyerscal}{\phi} 
\newcommand{\Scalfunc}{\Phi} 
\newcommand{\Corofunc}{W} 
\newcommand{\Coro}{\mathscr{K}} 
\newcommand{\conefunc}{v} 
\newcommand{\Conefunc}[1]{V_{(#1)}} 
\newcommand{\Cone}[1]{\mathscr{C}_{(#1)}} 
\newcommand{\pscal}{A} 
\newcommand{\pshear}{S} 
\newcommand{\pscalcone}[2]{A_{#1,(#2)}} 
\newcommand{\shearcone}[1]{S_{(#1)}} 
\newcommand{\unishplain}{\psi} 
\newcommand{\unishplainft}{\ft{\unishplain}} 
\newcommand{\unish}[5][{}]{\unishplain_{#3,#4,#5}^{#2\ifargdef{#1}{,(#1)}}} 
\newcommand{\unishft}[5][{}]{\unishplainft_{#3,#4,#5}^{#2\ifargdef{#1}{,(#1)}}} 
\newcommand{\Scalparamdomain}{\mathsf{A}} 
\newcommand{\aj}{\alpha_j} 
\newcommandx{\Unish}[3][1=\meyerscal,2=\conefunc,3=(\aj)_j]{\operatorname{SH}(#1, #2, #3)} 
\newcommandx{\UnishLow}[1][1=\meyerscal]{\operatorname{SH}_{\mathrm{Low}}(#1)} 
\newcommandx{\UnishInt}[3][1=\meyerscal,2=\conefunc,3=(\aj)_j]{\operatorname{SH}_{\mathrm{Int}}(#1, #2, #3)} 
\newcommandx{\UnishBound}[3][1=\meyerscal,2=\conefunc,3=(\aj)_j]{\operatorname{SH}_{\mathrm{Bound}}(#1, #2, #3)} 
\newcommand{\Unishgroup}{\Gamma} 
\newcommand{\Unishind}{\gamma} 
\newcommand{\lmax}{l_j} 
\newcommand{\weight}{w} 
\newcommand{\wlen}{\rho} 
\newcommand{\model}{{\weight\Linedistr}} 
\newcommand{\modelrec}{\model^\star} 
\newcommand{\Corofilter}{F} 
\newcommand{\mdiam}{h} 
\newcommand{\mask}[1]{\mathscr{M}_{#1}} 
\newcommand{\Unishshort}{\Psi} 
\newcommand{\scalpm}[1]{#1^{\pm1}} 
\newcommand{\translind}[2]{#1^{(#2)}} 

%% file: abstract.tex
\begin{abstract}

 This paper  studies the approximation of generalized ridge functions, namely of functions which are constant along some submanifolds of $\R^N$. We introduce the notion of linear-sleeve functions, whose function values only depend on the distance to some unknown linear subspace $L$. We propose two effective algorithms to approximate linear-sleeve functions $f(x)=g(\dist{x}{L}^2)$, when both the linear subspace $L\subset \R^N$ and the function $g\in C^s[0,1]$ are unknown. We will prove error bounds for both algorithms and provide an extensive numerical comparison of both. We further propose an approach of how to apply these algorithms to capture general sleeve functions, which are constant along some lower dimensional submanifolds.

\end{abstract}

%% file: contents2016.tex
\section{Introduction}
Nowadays we are living in a world where the acquisition, analysis and storage of big data play a major role. Digital communication, medical imaging, seismology and cosmology are only a few examples, which show the necessity to handle \sk{massive data sets}. Usually data is modeled as functions $f: X\rightarrow Y$, where $X$ can be $\mathbb{R}^N$ or a general curved surface. In particular the approximation of such functions from point queries, \sk{when $N$ is very large,} is an important field. Such problems arise, for example, in learning theory \cite{wain}, in modeling physical and biological systems \cite{Phys}, as well as neural networks \cite{Candes} and in parametric and stochastic PDEs \cite{Schwab}.

Because of the so-called curse of dimensionality, a notion introduced in 1961 by Richard Bellman \cite{Bellmann}, the handling of \sk{functions in many variables} is an ambitious task. Namely, functions on $\R^N$ with smoothness of order $s$ can in general be recovered with an accuracy of at most $n^{-s/N}$, applying $n$-dimensional spaces of linear or nonlinear approximation. Thus, the learning of functions depending on a large number of variables is particularly difficult even with smoothness assumptions on $f$ \cite{DeVoreLorentz,Nonlinear,Novak}. Certainly, we need to impose additional structure on $f$ to achieve efficient learning \cite{cevherVec,cevherMatr,daubechiesRidge,fornasierVybiral,kolleck}. 

\subsection{Ridge Functions}
One popular approach to break the curse of dimensionality is to consider ridge functions of the form 
\begin{align}\label{eqn:matrixridge}
 \R^N\supseteq \Omega \ni x \mapsto f(x)=g(Ax),
\end{align}
where $A\in \R^{m\times N}$, with $m$ considerably smaller than $N$, is usually called \emph{ridge matrix} \sk{and} $g\in C^s(\R^m)$, $1\!\le\!s\le\!2$, is called the \emph{ridge profile}. The requirement for the function to have at least one derivative is essential. In fact, it was shown in \cite{Entropy} that ridge functions need to have a first derivative uniformly bounded away from zero in the origin in order to reduce the complexity of the approximation task. 

 For particular choices of $A$ different approaches \sk{to successfully learn ridge functions} have been investigated. For example, if $A$ is of the form $A^T=[e_{i_1},\dots,e_{i_m}]$, for $e_{i_k}\in \R^N$ being the canonical unit vectors and $i_k\in\{1,\dots,N\}$, $f$ can be rewritten as a function which depends only on a few variables, i.e., $f(x_1,\dots,x_N)=g(x_{i_1},\dots,x_{i_m})$. An approach to recover the active variables and approximating the ridge profile $g$ has been given in \cite{Wojta}. It was shown that by adaptive sampling we can obtain similar estimates as if the active coordinates $i_1,\dots, i_m$ are known to us.
 


\sk{Another special case of \eqref{eqn:matrixridge}} is to assume that $m=1$  and that the matrix $A$ is, therefore, a vector, usually called \emph{ridge vector} and denoted by $A^T=:a$. In this case, $f$ is of the form
\begin{align}\label{eqn:ridge}
 f(x)=g(\langle x,a \rangle ).
\end{align}
The recovery of such ridge functions from point queries was first considered by Cohen, Daubechies, DeVore, Kerkyacharian, and Picard in \cite{daubechiesRidge} for ridge functions with a positive ridge vector.
It was shown that the accuracy of their method is close to the approximation rate of one-dimensional continuous functions.

However, the algorithm from \cite{daubechiesRidge} does not apply to arbitrary ridge vectors. In \cite{cevherVec,fornasierVybiral, kolleck} new algorithms were introduced to waive the assumption of a positive ridge vector.
The main idea of the algorithm in \cite{kolleck} is to approximate the gradient of $f$ by divided differences, exploiting the fact that the gradient of $f$ is some scalar multiple of the ridge vector. The accuracy of the approximation of the gradient is determined by the choice of the step size in the computation of the divided differences, whereas the number of sampling points is fixed.

The approach by Fornasier, Schnass and Vybiral \cite{fornasierVybiral} is rather based on compressed sensing \sk{and applies to \eqref{eqn:matrixridge} very generally}. Thus, not the gradient but the directional derivatives of $f$ were approximated at a certain number of random points in random directions.
However, especially for the methods in \cite{fornasierVybiral, daubechiesRidge}, the authors need a restrictive assumption to use compressed sensing techniques. That is, the ridge vector $a$ can be well-approximated by a sparse subset of its coefficients. In \cite{cevherVec} this  assumption could be removed by leveraging the Dantzig selector \cite{Dantzig} to recover an approximation of $a$.

However, the structure assumption on $f$ \sk{to be} a ridge function can be very restrictive. If we, for example, consider a sensor network, where we have a certain number of sensors, say $N$, which measure the moisture, temperature and pressure to forecast forest fire, the aim is to compute the risk of fire by a function $f:\R^{3N}\rightarrow \R$ depending on the measurements of the sensors. It is then very unlikely that the combination of measurements which yield the same risk of fire lie on a $3N-1$-dimensional hyperplane, since also parameters like topography and vegetation influence the prediction. Much more likely is the assumption that these combinations lie on a lower dimensional manifold.

\subsection{Sleeve Functions}
To allow for the recovery of more general functions, which are constant along some lower-dimensional submanifolds, we will introduce the notion of \emph{sleeve functions} and as a special case of \emph{linear-sleeve} functions.  \sk{Within this paper}  we will then investigate and analyze algorithms to capture linear-sleeve functions and we will propose a technique to apply these methods to general sleeve functions. 

\begin{definition}
 Let $g\in C^s[-r,r]$, $r\in \R_+$, $M\subset \R^N$ a \sk{$d$-dimensional} smooth submanifold of $\R^N$, and $\text{tub}_r(M):=\{x\in \R^N:\dist{x}{M}<r\}$, then we call $f:\text{tub}_r(M)\rightarrow \R$ a \emph{sleeve function} if we can rewrite $f$ in terms of $g$ by
 \begin{align}\label{eqn:manridge}
  f(x)=g(\dist{x}{M}^2),
 \end{align} 
 for $x\in \text{tub}_r(M)$. \sk{In the case where $M$ is a linear subspace, we call $f$ a \emph{linear-sleeve function} and denote $L:=M$, to emphasize the special case, i.e.,  we write:
 \begin{align}\label{eqn:linridge}
   f(x)=g(\dist{x}{L}^2).
 \end{align}}
\end{definition}

\sk{The need to restrict $g$ to a bounded domain is twofold; on the one hand, if we wish to recover $g$ from a finite number of sampling points, we do need this restriction and on the other hand, it is useful for the approximation task to have  a unique mapping $x\mapsto x_0$ with $\dist{x}{M}=\dist{x}{x_0}$.  Thus, in the case of $M$ being a linear subspace, $r$ can be chosen arbitrarily, where in the general case  $r$ is chosen to be the radius of a non-self-intersecting tube around M. For an illustration of linear-sleeve functions we refer to Figure \ref{fig:linsleeve}.}

Note that the notion of sleeve functions is indeed a generalization of ridge function, thus, if $L$ is an $N-1$-dimensional subspace, we can rewrite $\dist{x}{L}^2=\langle x,a\rangle^2$, where $a$ is the normal vector of $L$. Also note that this formulation is very different from the one introduced in \cite{fornasierVybiral}. Indeed, if $f$ is of the form $f=g(A\cdot)$, the level sets are linear subspaces, whereas this is not true for linear-sleeve functions (cf. Figure \ref{fig:linsleeve}).

Furthermore, observe that even by separating the approximation task in approximating $g$ and $M$, we cannot simply use manifold learning algorithms to approximate $M$, since manifold learning algorithms (cf. e.g. \cite{Mauro, Baraniuk}) usually assume that we can sample from the manifold. However, we need to reconstruct the level sets (or at least one, namely $M$) without knowing in advance to which level set the sampling points belong; actually it is very likely that all sampling points belong to different level sets.

\subsection{Our Contribution}
Our work studies the approximation of linear-sleeve function of the form $f(x)=g(\dist{x}{L})$ for $x\in \text{tub}_1(L)$, where $g\in C^s[0,1]$ and $L\subset R^N$ is a $d$-dimensional subspace of $\R^N$.  We will provide and analyze two different algorithms to capture linear sleeve functions from point queries. Our main contributions can be summarized as follows.
\begin{itemize}
 \item \emph{Adaptive Algorithm}. The first algorithm, to which we refer to as \emph{ATPE}, is based on the fact that the gradient of $f$ in some $x\in \R^N$ is perpendicular to the level set of $f$ in $x$. We will show that the restriction of $f$ to the plane, which is perpendicular to the gradient and which is the tangent plane  of the corresponding level set, is again a linear-sleeve function. We will then argue that applying the same fact iteratively to the restrictions of $f$, the tangent plane computed in the $N-d$-th step gives a reconstruction of $L$. In ATPE will then substitute the gradient by divided differences, because we cannot compute the gradient by point queries of $f$.  
 \item \emph{Optimization Algorithm} The second algorithm, to which we refer to as \emph{OGM}, is based on a minimization over the Grassmannian manifold. Namely, it will define an objective function, whose minimizer is $L$. However, we will see that we cannot define this objective functions using only point samples of $f$. In OGM we  therefore approximate this objective function by an objective function whose minimizer $\tilde{L}$ will be proven to be close to $L$.
 \item \emph{Error Bounds.} Those two algorithms are of a rather different nature. Whereas the approximation success of the first algorithm depends only on the error of the gradient approximation by divided differences, the success of the second algorithms depends on the error of the approximation of the objective function. 
 The first main theorem states that the error of the approximation of the $d$-dimensional subspace $L$ using ATPE can be bounded by
\begin{align}
 \|L-\tilde{L}\|_{\text{HS}}\le C(1+K)^{N-d}\sqrt{N-d}h^s,
\end{align}
where $\tilde{L}$ is the approximation of $L$, $h$ can be chosen arbitrarily small but fixed, $C,K$ are some positive constants and the number of function evaluations is given by $(N+1)(N-d)$. 
For the second main theorem we prove that using OGM the approximation error is given by 
\begin{align}
\|L-\tilde{L}\|_{\text{HS}}\le \tilde{C} \sqrt{N-d}M^{-1},
\end{align}
where $M$ is the number of function samples and $\tilde{C}$ a constant only depending polynomial on the dimension
of the space. Note that OGM, differently to ATPE, yields a reconstruction error which decreases with the number of sampling points and is not constrained by a fixed number of sampling points and is, therefore, advantageous. However, the first algorithm is more promising to apply also to the manifold case.
 \item \emph{Impact on the Approximation of General Sleeve Functions.} The next step would be to find algorithms to recover general sleeve functions of the type \eqref{eqn:manridge}. Due to the fact, that we would need to optimize over all possible $d$-dimensional submanifolds, to approximate general sleeve functions in a similar way as proposed by OGM, we anticipate that an adaptation of ATPE is more promising.

We believe that one can also use gradient approximations to capture general sleeve functions of the type \eqref{eqn:manridge}. Roughly said, we propose to use the gradients to compute samples from the manifold. More precisely, knowing the gradient of $f$ at some point $x$, again would enable us to approximate the sleeve profile $g$ and, under additional assumption, we could use the direction given by the gradient and the value of $f$ in $x$ to translate $x$ to the manifold. A careful estimation of the distribution of the translated sample points should then enable us to apply manifold learning algorithms (e.g \cite{Mauro}) to estimate the manifold $M$. 
 
\end{itemize}

The paper is organized as follows: After introducing some preliminaries, we will present and analyze ATPE in Section \ref{sec:Gradest}. In Section \ref{sec:Opt} we will introduce and analyze OGM.  The consideration will be completed in Section \ref{sec:num} by some promising numerical results.

\begin{figure}
\begin{center}\includegraphics[scale=0.6]{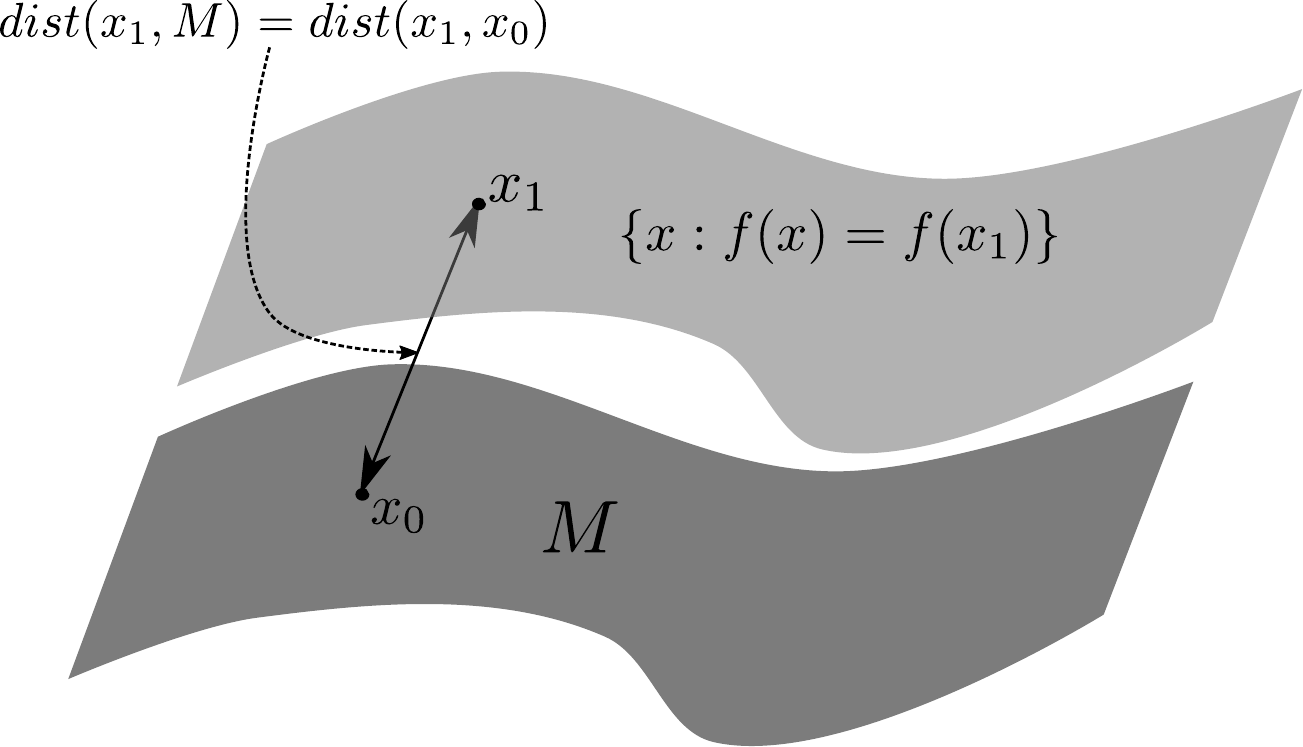}\end{center}\caption[Two level sets of a generalized ridge function of the form $f(x):=g(\dist{x}{M})$.]{Generalized ridge function of the form $f(x):=g(\dist{x}{M})$}
\end{figure}

\section{Preliminaries}\label{sec:pre}
To put our results in a precise setting, we introduce the class $\mathcal{LR}(s)$ of all linear-sleeve functions $f(x):=g(\dist{x}{L}^2)$, where $g\in C^s[0,1]$ and $L\subset \R^N$ a subspace. We use the following norm, subsequently referred to as \emph{H\"older norm}, on $C^s$. For $k< s\le k+1$, with $k\in \N$, we define
\begin{align}
 \|g\|_{C^s}:=\|g\|_{C^s[0,1]}:=|g^{(k)}|_{\text{Lip}(s-k)}+\sum_{j=0}^k\|g^{(j)}\|_{C[0,1]},
\end{align}
where $g^{(j)}$ denotes the $j$-th derivative of $g$, and, for $0<\beta\le 1$, we set
\begin{align}
 |g|_{\text{Lip}(\beta)}:=\sup_{x\neq y}\frac{\left|g(x)-g(y)\right|}{\left|x-y\right|^\beta}.
\end{align}
Note that we call $g$ \emph{Lipschitz continuous} if $|g|_{\text{Lip}(1)}$ is bounded. We then call $|g|_{\text{Lip}(1)}$ \emph{Lipschitz constant} or \emph{Lipschitz norm} of $g$.
If we want to highlight the dimension of the vector space, we sometimes write $\|\cdot\|_{\ell_p^N}:=\|\cdot\|_p$ for the $\ell_p$ norm of a vector, for $p=1,2$. The weak $\ell_p$ norm of a vector $x\in \R^N$ is the smallest constant $M$, such that
\begin{align}
 \#\{i:x_i \ge\varepsilon\}\le M\varepsilon^{-1/p}, \varepsilon>0. 
\end{align}
We further recall the following useful property of any norm on $\R^N$. 
\begin{lemma}[\cite{kolleck}]\label{lemma:kolleck}
 Let $\|\cdot\|$ be any norm on $\R^N$ and $x\in\R^N$ with $\|x\| = 1$, $\tilde{x} \in \R^N\setminus\{0\}$ and $\lambda \in \R$. Then
 \begin{align}
 \|\sign(\lambda)\frac{\tilde{x}}{\|\tilde{x}\|}-x\|\le \frac{2\|\tilde{x}-\lambda x\|}{\|\tilde{x}\|}.
 \end{align}
\end{lemma}

We will denote the $i$-th canonical unit vector, with a one in the $i$-th coordinate and zero elsewhere, by $e_i$. The Grassmannian manifold of all $d$-dimensional subspaces of $\R^N$ is denoted by $G(d,N)$ and for an orthogonal projection $P:\R^N\rightarrow \R^N$, the operator norm is given by the Hilbert-Schmidt norm
\begin{align}
 \|P\|_{\text{HS}}=\sqrt{\sum_{i=1}^N\|Pe_i\|^2_2}.
\end{align}
The orthogonal complement of a subspace $P$ is denoted by $P^{\perp}$ and the distance of a vector $x\in \R^N$ to a subspace, respective subset, $P\subset\R^N$ is defined by
\begin{align}
 \dist{x}{P}:=\min_{y\in P}\|x-y\|_2.
\end{align}
In the sequel, for two quantities $A,B \in \R$, which may depend on several parameters, we shall write $A\lesssim B$, if there exists a constant $C>0$ such that $A\le CB$, uniformly in the parameters. If the converse
inequality holds true, we write $A\gtrsim B$ and if both inequalities hold, we shall write $A\asymp B$.

Finally, we want to recall some approximation properties of functions in $C[0,1]$. For two given integers $S>1$ and $M\ge2$, we consider the space $\mathcal{S}_{h,S}$, $h:=1/M$, of piecewise polynomials of degree $S-1$ with equally spaced knots at the points $ih$, $i = 1, \dots , M - 1$, and having continuous derivatives of order $S - 2$. It is well-known (cf. e.g. \cite{DeVoreLorentz}) that there is a class of linear operators $Q_h$ which maps $C[0, 1]$ into $\mathcal{S}_{h,S}$. These operators are usually called \emph{quasi-interpolants}. For a function $g\in C[0,1]$, the application of a quasi-interpolant only depends on the values of $g$ at the points $ih$, $i = 0,\dots , M$. Furthermore, we can choose the operator $Q_h$ to fulfill the following  property: For all $g\in C^s([0,1])$ 
\begin{align}\label{Prob1}
\|g - Q_hg\|_{C[0,1]} \le C|g|_{C^s[0,1]}h^s,
\end{align}
  with $C$, a constant depending only on $S$ \cite{DeVoreLorentz}.


\section{An Adaptive Algorithm Estimating Tangent Planes}\label{sec:Gradest}
An obvious approach to approximate sleeve functions of the form \eqref{eqn:manridge} is to apply the methods to recover classical ridge functions of the form \eqref{eqn:ridge}, i.e., to approximate a linear sleeve function by a classical ridge function. Of course, this method can only provide good approximation results if the sleeve function is close to a classical ridge function in a certain sense, cf. \cite{Masterarbeit}. However, it would be more convenient to approximate a function of the form \eqref{eqn:manridge} by an estimator of the same form. 

\begin{figure}
\begin{center}\includegraphics[scale=0.4]{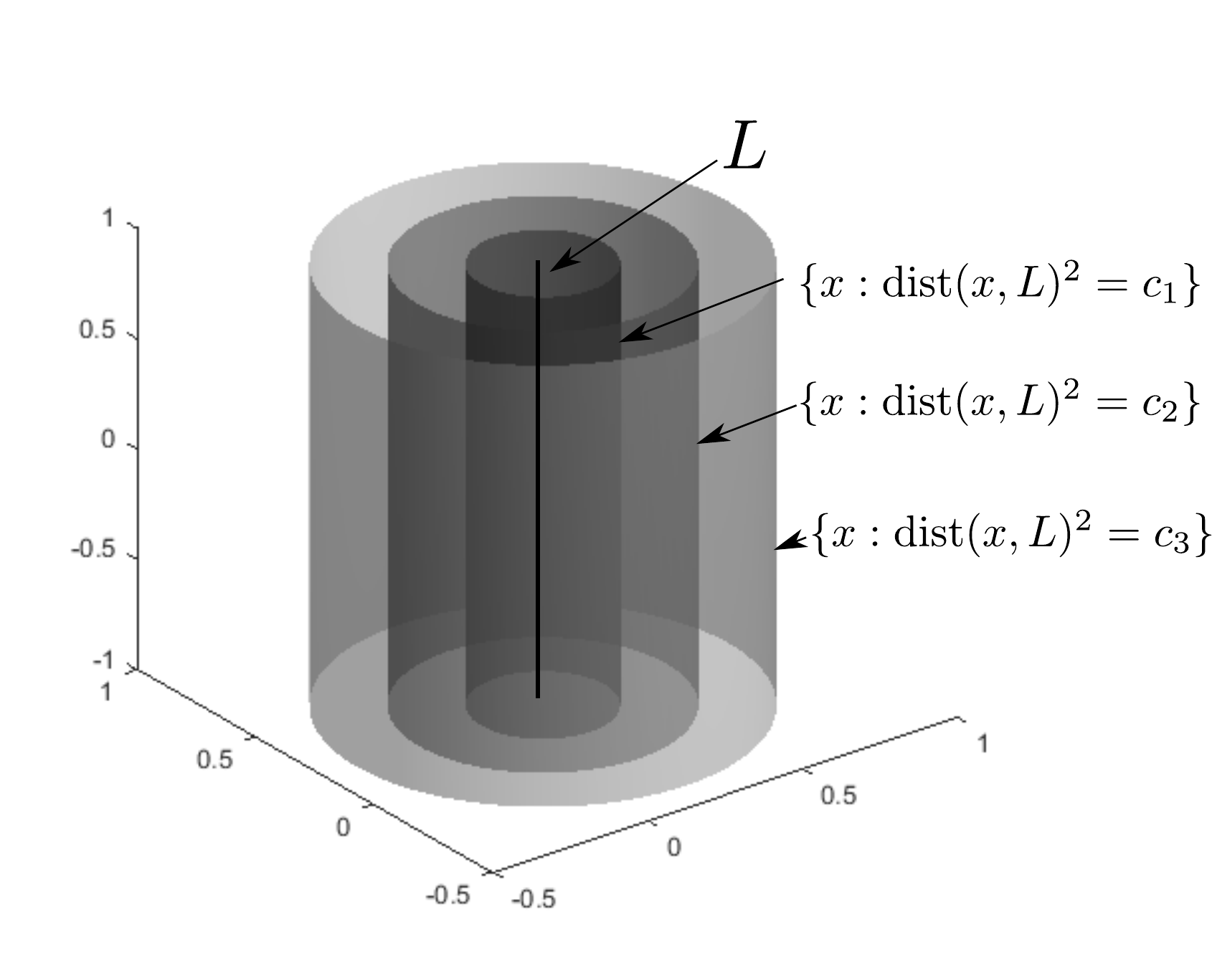}\hspace{1cm}\includegraphics[scale=0.65]{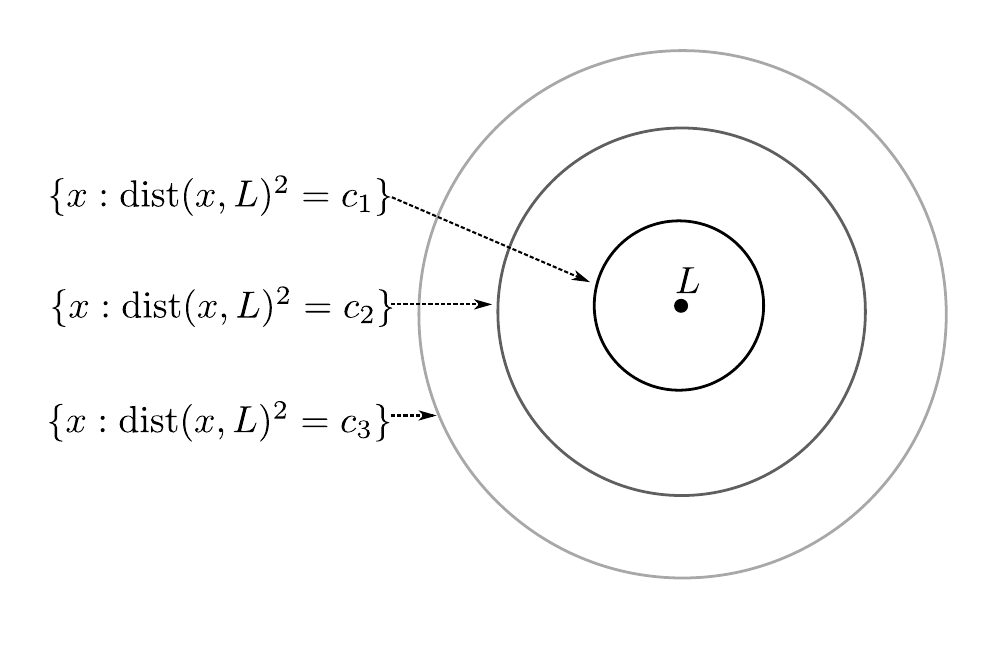}
\put(-338,-5){(a)}
\put(-65,-5){(b)}
\caption[Generalized ridge function of the form $f(x):=g(\dist{x}{L}^2)$.]{Generalized ridge function of the form $f(x):=g(\dist{x}{L}^2)$, where $L$ is some one-dimensional affine subspace of $\R^N$. The figures show three level sets for the function $f(x_1,x_2,x_3)=x_2^2+x_3^2=\dist{x}{L}^2$, where $L:=\spann{(1,0,0)}$ (blue line). (a) The level sets are illustrated in the ambient space $\R^3$. (b) Projection of the level sets onto the subspace orthogonal to $L$, which is here the $xy$-plane.}\label{fig:linsleeve}
\end{center}
\end{figure}

We first observe that we can rewrite a linear-sleeve function as
\begin{align}
 f(x)=g(\dist{x}{L}^2)=g(\|P_Px\|_2^2),
\end{align}
where $P_P$ is the orthogonal projection to the $(N-d)$-dimensional subspace $P$ orthogonal to $L$.  For simplicity we will denote the orthogonal projection of a vector $x\in \R^N$ to a subspace $H=\spann\{u_1,\dots,u_d\}\in G(d,N)$ by $Hx:=P_H(x)$. This notation relates to the matrix representation of an orthogonal projection given by $H=\sum_{i=1}^d u_iu_i^T$.

The algorithm, we will introduce in this section, will, similarly as in \cite{daubechiesRidge, kolleck}, exploit the fact
 that we can estimate the tangent plane in some $x_0\in \R^N$ of the $(N-1)$-dimensional submanifold
\begin{align*}
 \left\{x\in \R^N:\dist{x}{L}^2=\dist{x_0}{L}^2\right\}
\end{align*}
as the unique hyperplane which is perpendicular to the gradient of $f$ in $x_0$. We will then show that the function $f$ restricted to this tangent plane is again of the form \eqref{eqn:linridge}.
Of course, we cannot compute the gradient by sampling the function; in the subsequent proposed algorithm we, therefore, approximate the gradient by divided differences.

\begin{figure}[H]
\begin{center}
 \includegraphics[scale=0.5]{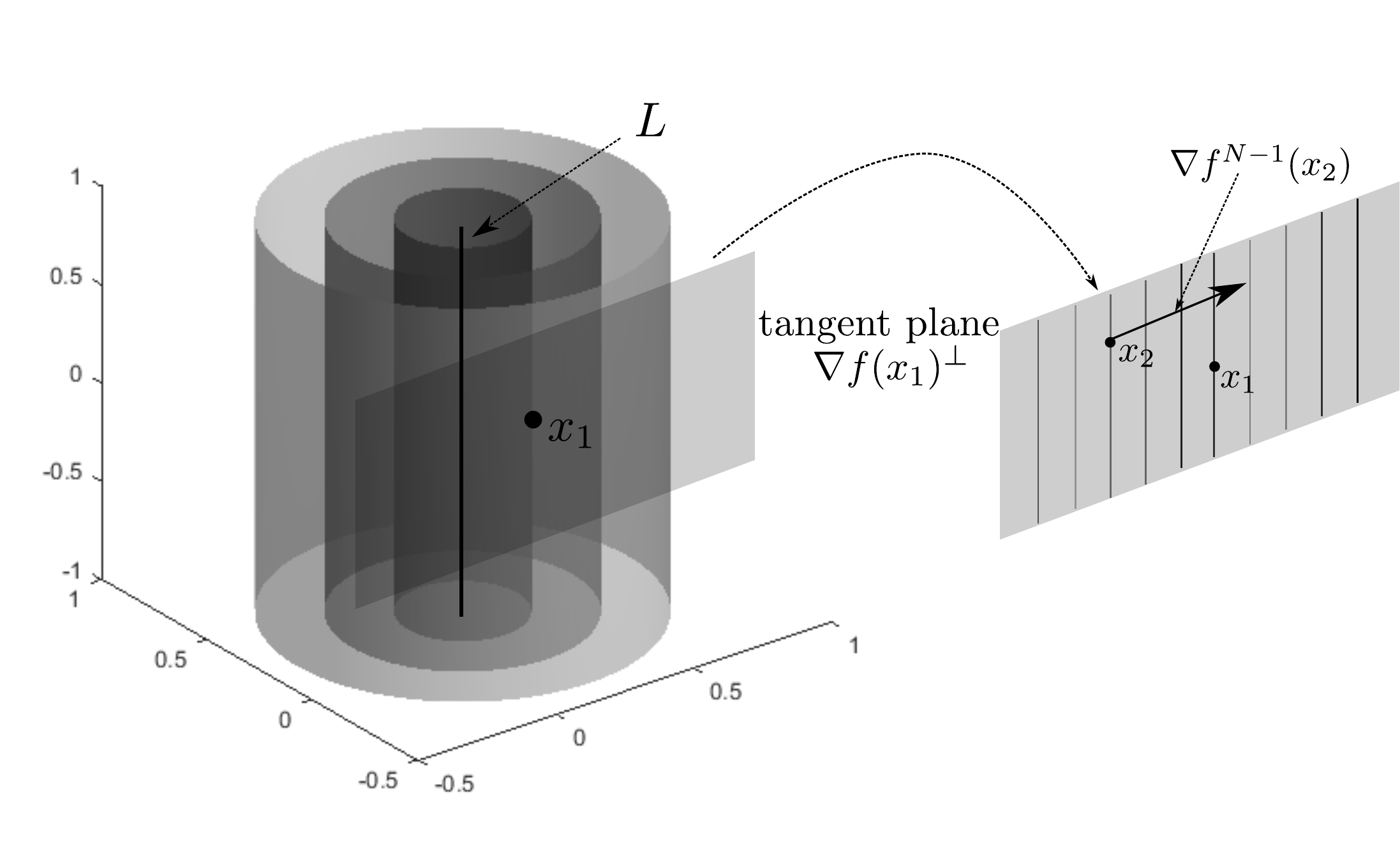}
 \caption{The restriction of $f$ to the affine subspace which is orthogonal to the gradient of $f$ in some $x_1$ is again a linear sleeve function. In the illustrated case we can recover $L$ after two steps, because the gradient of the restriction of $f$ in some $x_2$ is orthogonal to $L$.}\label{fig:GenLinAlg}
\end{center}
\end{figure}

\subsection{The Algorithm}

As mentioned before the idea of the first algorithm is to use the fact that the gradient of the linear-sleeve function $f$ in some $x_0$ is perpendicular to the level sets and that the restriction of $f$ to the corresponding tangent plane is again a linear sleeve function (cf. Figure \ref{fig:GenLinAlg}). We will further show that applying this fact iteratively to restrictions of $f$ finds after $N-d$ steps the wanted subspace $L$.  To prove this statement we introduce an adaptive algorithm to exactly recover the subspace $L$ by computing gradients of $f$ (see Algorithm \ref{GC}) and show in Theorem \ref{thm:gradest1} that this algorithm can recover the subspace $L$ exactly. Thus, our first aim is to show that the system, which is formed by the gradients of the restrictions of $f$, forms a basis for $P=L^{\perp}$. Observe that an essential step in this algorithm is to compute gradients of $f$ and is therefore not useful to approximate $f$ from point samples. It only serves as an auxiliary tool to introduce the first main Algorithm \ref{GE}.\\

\begin{algorithm}[H]
 \DontPrintSemicolon
\KwData{$f(\cdot)=g(\dist{\cdot}{L}^2)$.}
\KwResult{$T^d$.}
\Begin{
  $f^N \longleftarrow f$.\;
  $T^N \longleftarrow \R^N$.\;
  \For{$i=N,\dots,d+1$}{
    1. For some arbitrarily chosen normalized $x_i\in T^i$ compute $\nabla{f^i}(x_i)$.\;
2. $u_i\longleftarrow\nabla f^i(x_i)/\|\nabla f^i(x_i)\|_2$.\;
3. $T^{i-1}\longleftarrow\left(\spann\{u_i,\dots,u_N\}\right)^{\perp}$.\;
4. Let $f^{i-1}$ be the restriction of $f$ to $T^{i-1}$.\;}
    }
\caption{ATPC - Approximation by \underline{A}daptive \underline{T}angent \underline{P}lane \underline{C}omputation\label{GC}}
\end{algorithm}

As mentioned before this algorithm iteratively computes restrictions of $f$, such that after $N-d$ steps the restriction of $f$ will be exactly defined on $L$. Assume we have computed the tangent plane $T^i$ and the restriction $f^i$ of $f$ the subspace $T^i\subset\R^N$, for $i=N,\dots,d+1$. The algorithm then chooses uniformly at random a point $x_i\in T^i$ and computes the gradient of $f^i$ in $x_i$ (1.). It then normalizes this gradient (2.) and determines the subspace $T^{i-1}$ which is orthogonal to this gradient in $T^i$ (3.). Finally we restrict $f^i$ to $T^{i-1}$ and repeat this procedure until $i=d+1$. The following theorem states that ATPC indeed succeeds to recover $L$.

\begin{theorem}\label{thm:gradest1}
Let $f(x)=g(\dist{x}{L}^2)=g(\|Px\|_2^2)$, $g\in C^s[0,1]$, $1< s\le 2$, for $x\in \text{tub}_1(L)$ and the linear subspace $P^{\perp}=L\subset \R^N$. Compute $T^d$ as proposed in Algorithm \ref{GC}. Then $L$ coincides with $T^d$ almost surely (e.g. due to the Lebesque measure on $\R^N$).
\end{theorem}
 \begin{proof} We write $L$ as  $L=\spann\left\{u_1,\dots,u_d\right\}$ where $\left\{u_1,\dots,u_d,u_{d+1},\dots,u_N\right\}$ is an orthonormal basis of $\R^N$, and let $V=\left[u_1\dots u_d\, u_{d+1}\dots u_N\right]$ be the corresponding matrix.
We begin by computing the gradient of $f$ and obtain
\begin{equation*}
\nabla{f}(x_N)=2g'(\|Px_N\|_2^2)Px_N,
\end{equation*} 
which is obviously perpendicular to $L$ if $Px_N\neq 0$, which is almost surely true. 

Due to the orthogonality of $\nabla{f}(x_N)$ to $L$, we can assume that $u_N\!=\!\nabla{f}(x_N)/\|\nabla{f}(x_N)\|_2$ and then by definition $T^{N-1}=(\spann\{u_N\})^\perp=\spann\{u_1,\dots,u_{N-1}\}$. We now define $f^{N-1}$ to be the restriction of $f$ to $T^{N-1}$ and $h^{N-1}:\R^{N-1}\rightarrow \R$ by 
\begin{align}
h^{N-1}(\hat{x}):=f(V\begin{bmatrix}
\hat{x}\\0
\end{bmatrix})=f^{N-1}(V\begin{bmatrix}
\hat{x}\\0
\end{bmatrix}),
\end{align}
for $\hat{x}\in \R^{N-1}$. Then $\nabla{h^{N-1}}^T(\hat{x})=\nabla{f}(x)^T\hat{V}_{N-1}$, where \sk{$\hat{V}_i:=\left[u_1\dots u_i\right]$, $i=1,\dots, N$,} and $x=:V(\hat{x},0)^T\in T^{N-1}$. Thus, the gradient of $h^{N-1}$, considered as a vector in $\R^N$ is given by $$\begin{bmatrix}
\nabla{h^{N-1}}(\hat{x})\\0
\end{bmatrix}=V_{N-1}^T \nabla{f}(x),$$ where \sk{$V_i:=\left[u_1\dots u_i 0\dots 0\right]$.} We conclude that for some  $x_{N-1}\in T^{N-1}$ chosen uniformly at random, we have
\begin{align}\label{est:nablarest}
\nabla{f}^{N-1}(x_{N-1})=VV_{N-1}^T\nabla{f}(x_{N-1})=2g'(\|Px_{N-1}\|_2^2)VV_{N-1}^TPx.
\end{align}

A straightforward computation shows that $VV_{N-1}^TPx$ is the projection of $Px$ to $T^{N-1}$. Thus, it is obvious that for some $x_{N-1}\in T^{N-1}$, chosen as in Algorithm \ref{GC}, $\nabla{f}^{N-1}(x_{N-1})$ is perpendicular to $L$, i.e., if $Px_{N-1}\neq 0$, which is almost surely the case.

Further,  $\nabla{f}^{N-1}(x_{N-1})$ is also orthogonal to $u_N$, since it lies in $T^{N-1}$.
Therefore, we set $$u_{N-1}=\nabla{f}^{N-1}(x_{N-1})/\|\nabla{f}^{N-1}(x_{N-1})\|_2$$ for some $x_{N-1}\in T^{N-1}$ and $T^{N-2}:=\spann\left\{u_1,\dots,u_{N-2}\right\}$. Note that again $Px_{N-1}\neq 0$ holds almost surely. We repeat this procedure until we get a basis $\left\{u_{d+1},\dots,u_N\right\}$ of $P$, which yields the desired space $L=P^{\perp}$.
\end{proof}

The  previous theorem shows that, if we could compute the gradient in $(N-d)$ points, we would be able to recover the space $L$, respectively its orthogonal complement $P$, exactly. However, ATPC is not based on sampling the function $f$, since gradients cannot be computed exactly using point queries. Thus, we can only approximate the gradients by computing the divided differences
\begin{equation*}
\nabla_hf(x)=\left[\frac{f(x+he_i)-f(x)}{h}\right]_{i=1}^N.
\end{equation*}
We adapt the first step in ATPC by substituting the computation of the gradients by computing divided difference and  propose Algorithm \ref{GE}, to which we will refer as \emph{ATPE}, for the approximation task.\\

\begin{algorithm}[H]
\DontPrintSemicolon
\KwData{$f(\cdot)=g(\dist{\cdot}{L}^2)$.}
\KwResult{$\tilde{L}$.}
\Begin{
  $\tilde{f}^N \longleftarrow f$.\;
  $\tilde{T}^N \longleftarrow \R^N$.\;
  \For{$i=N,\dots,d+1$}{
    1. For some arbitrarily chosen normalized $\tilde{x}_i\in \tilde{T^i}$ compute
 \begin{align*}\nabla_h{\tilde{f}^i}(\tilde{x}_i)=\left[\frac{\tilde{f}^{i}(\tilde{x}_i+he_k)-\tilde{f}^{i}(\tilde{x}_i)}{h}\right]_{k=1}^N.\end{align*}\;
2. $\tilde{u}_i\longleftarrow\nabla_h\tilde{f}^i(\tilde{x}_i)/\|\nabla_h\tilde{f}^i(\tilde{x}_i)\|_2$.\;
3. $\tilde{T}^{i-1}\longleftarrow\left(\spann\{\tilde{u}_i,\dots,\tilde{u}_N\}\right)^{\perp}$.\;
4. Let $\tilde{f}^{i-1}$ be the restriction of $f$ to $\tilde{T}^{i-1}$.\;}
   $\tilde{L}\longleftarrow \tilde{T}^d$.\;
    }
\caption{ATPE - Approximation by \underline{A}daptive \underline{T}angent \underline{P}lane \underline{E}stimation \label{GE}}

\end{algorithm}\

The described procedure, of course, cannot find the correct plane $P$. However, it is able to compute a good approximation of $P$, where the approximation error depends on the choice of $h$. For reasons of clarity, the proof of the next theorem is \sk{moved} to the next subsection.

\begin{theorem}\label{thm:gradest} Let $f$ be a linear-sleeve function of the form \eqref{eqn:linridge}, i.e., $f\in \mathcal{LR}(s)$ for some $s\in(1,2]$.  Assume that the derivative of $g\in C^s([0,1])$ is bounded by some positive constants $c_2,c_3$.
By sampling the function $f$ at $(N-d)(N+1)$ appropriate points, ATPE constructs almost surely an approximation of $L$ by a subspace $\tilde{L}\subset \R^N$, such that the error is bounded by
\begin{equation*}
\|L-\tilde{L}\|_{\text{HS}}\le C\hat{C}(1+\hat{C}K)^{N-d}\sqrt{N-d}h^{s/2(s-1)^i},
\end{equation*}
for some arbitrarily small $h>0$, where 
\begin{align}
K=4|g'|_{C^s}+2c_3+\max\{\|\nabla f(x_{i})\|:i=N,\dots, d+1\}\\
\hat{C}=2/\min\{\|\nabla f^{N-i}(x_i)\|_2:i=N,\dots, d+1\}\\
C^2=h^{2-s}c_3+4|g'|_C^sh^{s-1}(2+h)^{2(s-1)}+c_3|g'|_{C^s}(2+h)^{s-1},\end{align}
which are constants depending only on the H\"older norm and bounds of $g'$.

In particular if $g'$ is Lipschitz continuous, i.e. $s=2$, it holds
\begin{equation*}
\|L-\tilde{L}\|_{\text{HS}}\lesssim (1+\tilde{C})^{N-d}\sqrt{N-d}h,
\end{equation*}
with $\tilde{C}=\hat{C}K$.
\end{theorem}

It then only remains to recover the ridge profile $g$. The estimation of $g$ is \sk{rather straightforward}. As the gradient gives the direction in which $f$ changes, f becomes a one-dimensional function in the direction of the gradient. Hence, we can estimate $g$ with well-known numerical methods. Indeed, we have already seen that the gradient of $f$ in some point $x$ is given by $\nabla f(x)=g'(\|Px\|_2^2)Px$, i.e., the normalized direction is $a:=Px/\|Px\|_2$. Setting $x_t:=ta$ yields
\begin{align}
 f(x_t)=g(\|Px_t\|_2^2)=g(\frac{t^2}{\|Px\|_2^2}\|Px\|_2^2)=g(t^2).
\end{align}

However, similar to the algorithm in \cite{kolleck}, this algorithm uses a fixed number of samples and the estimation cannot be improved by taking more samples. We, therefore, also aim for an algorithm which yields a reconstruction whose error decreases with the number of sampling points (cf. Section \ref{sec:Opt}).
Also note that due to the adaptive character of ATPE the reconstruction error increases for smaller values of the Lipschitz continuity $s$.

To complete this subsection, we want to remark that we can perform a similar, but slightly worse, error analysis for the case that $f$ is of the form $f(x)=g(\|Px\|_2)$, whereas before we considered sleeve functions of the form $f(x)=g(\|Px\|^2_2)$.  See Remark \ref{rem:dist} in the next subsection for a short explanation.

\subsection{Proof of Theorem \ref{thm:gradest}}
As mentioned above, the idea is to approximate the gradients of $f=f^N$ and $f^{i}$ for $i=d+1,\dots,N-1$. Since we need $N+1$ samples for each gradient approximation, we need $(N-d)(N+1)$ samples altogether.  
We already know from Theorem \ref{thm:gradest1} that the subspace $P$ can be written in terms of the gradients of the restrictions of $f$ and $f$ itself. Hence, we assume $P=L^{\perp}=\spann\left\{u_{d+1}\dots,u_N\right\}$, where the $u_i$'s are given as stated in Theorem \ref{thm:gradest1}. We split the proof by establishing several lemmata.

\begin{lemma}\label{claim:S0} Under the assumptions of Theorem \ref{thm:gradest} and with the choice of the $\tilde{u}_i$'s, $i=d+1,\dots, N$, as proposed in Algorithm \ref{GE}, we have
 \begin{align}\|\tilde{u}_N-u_N\|_2\le\frac{2C\sqrt{N-d}h^{s/2}}{\nabla f(\tilde{x}_N)}=:S_0,\label{est:TN}\end{align}
 where $\tilde{x}_N\in \R^N$ is chosen uniformly at random $u_N=\frac{\nabla f(\tilde{x}_N)}{\|\nabla f(\tilde{x}_N)\|_2}$.
\end{lemma}

\begin{proof}
First, let us estimate the error between $\nabla f(x)$ and $\nabla_h f(x)$. We can compute
\begin{align*}
\nabla_hf(x)_i&=\frac{f(x+he_i)-f(x)}{h}=\frac{g(\|P(x+he_i)\|_2^2)-g(\|Px\|_2^2)}{h}\\
&=g'(\xi_{i,h})\frac{\|P(x+he_i)\|_2^2-\|Px\|_2^2}{h}\\
&=g'(\xi_{i,h})\frac{\sum_{j=d+1}^N \langle x+he_i,u_j\rangle^2-\langle x,u_j\rangle^2}{h}\\
&=g'(\xi_{i,h})\sum_{j=d+1}^N 2u_{ji}\langle x,u_j\rangle+hu_{ji}^2\\
&=g'(\xi_{i,h})\left(2[Px]_i+h\sum_{j=d+1}^Nu_{ji}^2\right),
\end{align*}
for some $\xi_{i,h}$ between $\|Px\|_2^2$ and $\|P(x+he_i)\|_2^2$, where $u_{ji}$ denotes the $i$-th entry of the vector $u_j$. We then estimate
\begin{equation*}
\left|\xi_{i,h}-\|Px\|_2^2\right|\le |\|P(x+he_i)\|_2^2-\|Px\|_2^2|\le\sum_{j=d+1}^N\left|\left(2h\langle x,u_j\rangle u_{ji}+h^2u_{ji}^2\right)\right|=2h\left|[Px]_i\right|+h^2\sum_{j=d+1}^Nu_{ji}^2\le 2h+h^2,
\end{equation*}
where we used the fact that $u$ is a unit vector and that, therefore, all its entries (in absolute value) and the entries of its projection are smaller than or equal to one.
Thus, the error which we obtain by approximating the gradient can be estimated as
\begin{align*}
\|\nabla f(x)-\nabla_hf(x)\|_2^2&=\sum_{i=1}^Ng'(\xi_{i,h})^2h^2\left(\sum_{j=d+1}^Nu_{i,j}^2\right)^2\\
&+4\sum_{i=1}^N\left(g'(\xi_{i,h})-g'(\|Px\|_2)\right)^2(Px)_i^2\\
&+2h\sum_{i=1}^N\left|g'(\xi_{i,h})\right|\left|g'(\xi_{i,h})-g'(\|Px\|_2^2)\right|\left|(Px)_i\right|\sum_{j=d+1}^Nu_{ji}^2\\
&=:T_1+T_2+T_3.
\end{align*}
To estimate those terms we take the following inequality into account:
\begin{align*}
\sum_{i=1}^N\left(\sum_{j=d+1}^Nu_{ji}^2\right)^2= \sum_{i=1}^N\left(\sum_{j=d+1}^Nu_{ji}^2\right)\left(\sum_{j=d+1}^Nu_{ji}^2\right)\le \sum_{i=1}^N\left(\sum_{j=d+1}^Nu_{ji}^2\right)=\sum_{j=d+1}^N\left(\sum_{i=1}^Nu_{ji}^2\right)=N-d,
\end{align*}
where we used in the second as well as in the last step that $\{u_j\}_{j=1}^d$ forms an orthonormal system, which spans $H$, so that $\sum_{j=1}^du_{ji}^2\le1$ for each $i=1,\dots,N$ and $\sum_{i=1}^Nu_{ji}^2=1$ for each $j=1,\dots,d$. Now the desired estimates follow immediately:
\begin{align*}
T_1&\le h^2\|g'\|_{\infty}^2\sum_{i=1}^N\left(\sum_{j=1}^du_{ji}^2\right)^2\le (N-d)\|g'\|_{\infty}^2h^2,\\
T_2&\le4|g'|^2_{C^s}(2h+h^2)^{2(s-1)}\sum_{i=1}^N(Px)_i^2\le |g'|^2_{C^s}(2h+h^2)^{2(s-1)},\\
T_3&\le \|g'\|_{\infty}\|g\|_{C^s}(N-d)h(2h+h^2)^{s-1}.
\end{align*}
Thus, we can find a constant $C>0$, independent of the dimensions $d$ and $N$, such that
\begin{equation}\label{eqn:aTN-1}
\|\nabla{f}(x)-\nabla_hf(x)\|_2\le C\sqrt{N-d}h^{s/2}.
\end{equation} 
For the exact choice of $C$ we refer to Theorem \ref{thm:gradest}. Hence, applying Lemma \ref{lemma:kolleck}, with $\lambda = 1/\|\nabla_h f(x)\|_2$, proves the claim.
\end{proof}
Next, we use the approximation of the gradient to approximate the tangent plane $T^{N-1}$ at $x$ with $\tilde{T}^{N-1}=\spann\{\nabla_hf(x)^{\perp}\}$. The approximation error is then, of course, given by \eqref{est:TN}. Further, we let $f^{N-1}$ and $\tilde{f}^{N-1}$ be the restriction of $f$ to $T^{N-1}$ and $\tilde{T}^{N-1}$, respectively. 

Again we want to compute the column vectors $u_i$ of $V$, $i=N,\dots,d+1$, step by step as the normalized gradients of $f$, $f^i$.
But instead of computing the gradient of $f^{j}$ we can only approximate it through an approximation of the gradient of $\tilde{f}^{j}$. Thus, we iteratively set the columns $\tilde{u}_i$, $i=N,\dots,d+1$ of $\tilde{V}$ as the normalized approximated gradients of $\tilde{f}_i$. The error of the approximation in each step can then be estimated by means of the following lemmata, in particular, by means of Lemma \ref{claim:S1} for the  first step and Lemma \ref{claim:Si} for the $i$-th step.

Before stating these lemmata, let us recall the definition of the matrices $V, \tilde{V}, V_{i}$ and $\tilde{V}_i$, $i=N-1,\dots, d+1$. Let
\begin{align}\tilde{u}_{i}=\nabla_h\tilde{f}^{i}(\tilde{x}_{i})/\|\nabla_h\tilde{f}^{i}(\tilde{x}_{i})\|_2\quad  \text{and}\quad u_{i}=\nabla_hf^{i}(x_{i})/\|\nabla_hf^{i}(x_{i})\|_2,\end{align}
for $i=N,\dots,d+1$, where $x_{i}=VV_{i}^T\tilde{x}_{i}$ and $\tilde{x}_i$ as well as $f^i$ and $\tilde{f}^i$ are chosen as proposed in the Algorithms ATPC and ATPE. Then $\{u_{d+1},\dots, u_N\}$ as well as $\{\tilde{u}_{d+1},\dots, \tilde{u}_N\}$ form orthonormal systems and $u_i,\tilde{u}_i$, $i=1,\dots,d$, are chosen that the whole systems $\{u_{1},\dots, u_N\}$ and $\{\tilde{u}_{1},\dots, \tilde{u}_N\}$ form an orthonormal basis for $\R^N$. We can now define 
\begin{align}
 V_{i}&=\begin{bmatrix}
          u_1 & u_2 &\dots &u_{i}&0&\dots&0
         \end{bmatrix},\hspace*{-2cm}&
          V&=\begin{bmatrix}
          u_1 & u_2 &\dots &u_N
         \end{bmatrix},\\
\tilde{V}_{i}&=\begin{bmatrix}
          \tilde{u}_1 & \tilde{u}_2 &\dots &\tilde{u}_{i}&0&\dots&0
         \end{bmatrix},\hspace*{-2cm}&
         \tilde{V}&=\begin{bmatrix}
          \tilde{u}_1 & \tilde{u}_2 &\dots &\tilde{u}_N
         \end{bmatrix}.
\end{align}

\begin{lemma}\label{claim:S1} With the same assumptions and choices as in Theorem \ref{thm:gradest} and Lemma \ref{claim:S0}, we have
 \begin{align}
  \|\tilde{u}_{N-1}-u_{N-1}\|_2\le \hat{C}(1+K)S_0^{s-1}.
 \end{align}
\end{lemma}
We first have to prove the following lemma:
\begin{lemma}\label{claim:PPx} With the same assumptions and choices as in Theorem \ref{thm:gradest} and Lemma \ref{claim:S0},
 let $x:=VV_{N-1}^T\tilde{x}$, for some $\tilde{x}\in \tilde{T}^{N-1}$. We then have
 \begin{align}
  \|Px-P\tilde{x}\|_2\le \|\tilde{x}\|_2\|u_N-\tilde{u}_N\|_2.
 \end{align}
\end{lemma}
\begin{proof}
We write $\tilde{x}=\sum_{i=1}^{N-1}\tilde{x}_i\tilde{u}_i$ with $\tilde{x}_i\in \R$ and $|\tilde{x}_i|\le 1$. Then we compute
\begin{align}
\|Px-P\tilde{x}\|^2_2&=\|\sum_{i=d+1}^N\langle \sum_{j=1}^{N-1}\langle \tilde{x},u_j\rangle u_j,u_i\rangle u_i-\langle \tilde{x},u_i\rangle u_i\|_2^2
=\|\langle \tilde{x},u_N\rangle u_N\|_2^2=\|\langle \sum_{i=1}^{N-1} \tilde{x}_i\tilde{u}_i,u_N\rangle u_N\|_2^2\\
&= \left(\sum_{i=1}^{N-1} \tilde{x}_i\langle\tilde{u}_i,u_N\rangle\right)^2\overset{CS}{\le}\sum_{i=1}^{N-1}\tilde{x}_i^2\sum_{i=1}^{N-1}\langle \tilde{u}_i,u_N\rangle^2\\
&=\left(\sum_{i=1}^{N-1}\tilde{x}_i^2\right)\left(1-\sum_{i=1}^{N}\langle \tilde{u}_i,u_N\rangle^2+\sum_{i=1}^{N-1}\langle \tilde{u}_i,u_N\rangle^2\right)\\
&\le 2\|\tilde{x}\|_2^2\left(1-\langle \tilde{u}_N,u_N\rangle\right).
 \end{align}
 In the step before the last we used that $\{\tilde{u}_1,\dots,\tilde{u}_N\}$ is an orthonormal basis (according to Step 2. and 3. in the ATPE Algorithm) and in the last step additionally that $1-\langle\tilde{u}_N,u_N\rangle^2=(1+\langle\tilde{u}_N,u_N\rangle)(1-\langle\tilde{u}_N,u_N\rangle)\le 2(1-\langle\tilde{u}_N,u_N\rangle)$.
 By observing that
 \begin{align}
  \|u_N-\tilde{u}_N\|_2^2=\langle u_N-\tilde{u}_N,u_N-\tilde{u}_N\rangle=\langle u_N,u_N\rangle -2\langle u_N,\tilde{u}_N\rangle+\langle \tilde{u}_N,\tilde{u}_N\rangle=2\left(1-\langle u_N,\tilde{u}_N\rangle\right),
 \end{align}
 we deduce the claim.
\end{proof}
We are now able to  prove the error of the first step of our algorithm ATPE, i.e., Lemma \ref{claim:S1}.
\begin{proof}[Proof of Lemma \ref{claim:S1}]

For simplicity we will write $\tilde{x}$ and $x$ instead of $\tilde{x}_{N-1}$ and $x_{N-1}$. Note that it is straightforward to show (compare to Equation \eqref{est:nablarest}) that $\nabla f^i(x)=VV_i^T\nabla f(x)$ and $\nabla_h f^{i}(x)=\tilde{V}\tilde{V}_i\nabla_hf$.   We, hence, obtain that
\begin{align} \|u_{N-1}-\tilde{u}_{N-1}\|_2=&\left\Vert\frac{\nabla f^{N-1}(x)}{\|\nabla f^{N-1}(x)\|_2}-\frac{\nabla_h \tilde{f}^{N-1}(\tilde{x})}{\|\nabla_h \tilde{f}^{N-1}(\tilde{x})\|_2}\right\Vert_2
\le\frac{2\|\nabla f^{N-1}(x)-\nabla_h \tilde{f}^{N-1}(\tilde{x})\|_2}{\|\nabla f^{N-1}(x)\|_2}\\
\le&\frac{2}{\|\nabla f^{N-1}(x)\|_2}\left(\|VV_{N-1}^T\nabla f(x)-\tilde{V}\tilde{V}_{N-1}^T\nabla f(x)\|_2\right.\\&+\|\tilde{V}\tilde{V}_{N-1}^T\nabla f(x)-\tilde{V}\tilde{V}_{N-1}^T\nabla f(\tilde{x})\|_2+\left.\|\tilde{V}\tilde{V}_{N-1}^T\nabla f(\tilde{x})-\tilde{V}\tilde{V}_{N-1}^T\nabla_h f(\tilde{x})\|_2\right)\\
=:&\frac{2}{\|\nabla f^{N-1}(x)\|_2}\left(T_1+T_2+T_3\right),
\end{align}
where we applied Lemma \ref{lemma:kolleck} with $\lambda=\|\nabla_h\tilde{f}^{N-1}(\tilde{x})\|_2$ in the second step.
Now, we can estimate
\begin{align}
 T_1&=\|\sum_{i=1}^{N-1}\langle u_i,\nabla f(x)\rangle u_i-\langle \tilde{u}_i,\nabla f(x)\rangle \tilde{u}_i\|_2=\|\langle u_N,\nabla f(x)\rangle u_N-\langle \tilde{u}_N,\nabla f(x)\rangle \tilde{u}_N\|_2\le 2\|\nabla f(x)\|_2\|u_N-\tilde{u}_N\|_2
 \end{align}
 as well as
 \begin{align}
 T_3&\le \|\nabla f(\tilde{x})-\nabla_hf(\tilde{x})\|_2\le 2C\sqrt{N-d}h^{s/2}=S_0.
\end{align}
Finally, we estimate $T_2$ by
\begin{align}
 T_2&\le\|\nabla f(x)-\nabla f(\tilde{x})\|_2 =2\left\|g'(\|Px\|_2^2)Px-g'(\|P\tilde{x}\|_2^2)P\tilde{x}\right\|_2\\
 &\le2 \Big[\left|g'(\|Px\|_2^2)-g'(\|P\tilde{x}\|_2^2)\right|\|Px\|_2+\left|g'(\|P\tilde{x}\|_2^2)\right|\|Px-P\tilde{x}\|_2\Big]\\
 &\le 2|g'|_{C^s}\left|\|Px\|_2^2-\|P\tilde{x}\|_2^2\right|^{s-1}+2\max_{|t|\le 1}\{g'(t)\}\|Px-P\tilde{x}\|_2 \\
 &\le 4|g'|_{C^s}\|Px-P\tilde{x}\|_2^{s-1}+2\|g\|_{\infty}\|Px-P\tilde{x}\|_2,\\
\end{align}
which proves the lemma.
\end{proof}

One can now prove similar estimates as in Lemma \ref{claim:S1} for $\|u_i-\tilde{u}_i\|_2$, $i=d+1,\dots,N-2$. However, we first need to prove a more general version of Lemma \ref{claim:PPx}.
\begin{lemma}\label{claim:PPxi} With the same assumptions and choices as in Theorem \ref{thm:gradest} and Lemma \ref{claim:S0}, we have for $\tilde{x}_i\in \tilde{T}^{i}$ and $x=VV^T_i\tilde{x}_i$, $i=N-2,\dots,d+1$, that
 \begin{align}
  \|Px_i-P\tilde{x}_i\|_2^2\le \|\tilde{x}_i\|_2^2 \sum_{j=i+1}^N\|\tilde{u}_j-u_j\|_2^2.
 \end{align}

\end{lemma}
This inequality in turn yields the desired generalization of Lemma \ref{claim:S1}:
\begin{lemma} \label{claim:Si}With the same assumption and choices as in Theorem \ref{thm:gradest} and Lemma \ref{claim:S0}, for $i=0,\dots, d$, we have
 \begin{align}
\|u_{N-i}-\tilde{u}_{N-i}\|_2\le C(1+K\hat{C})^{i}S_0^{(s-1)^i},
 \end{align}
 for all $i=0,\dots,N-d-1$.

\end{lemma}
\begin{proof} Using the same methods as in the proof of Lemma \ref{claim:S1} and using Lemma \ref{claim:PPxi}, one can show
\begin{align}
 \|u_{N-i}-\tilde{u}_{N-i}\|_2\le \hat{C}K(\sum_{j=0}^{i-1}S_j)^{s-1}+\hat{C}S_0,
\end{align}
where the $S_j$, $j=1,\dots,d$ are recursively defined. The claim then follows by induction.
\end{proof}

Putting the conclusions of the previous lemmata together and observing that $$\sum_{i=0}^{N-d-1}(1+K\hat{C})^i\le(1+K\hat{C})^{N-d}$$ finishes the proof of Theorem \ref{thm:gradest}.
\qed

As mentioned in the end of the last subsection, we can perform a similar, but slightly worse, error analysis for the case that $f$ is of the form $f(x)=g(\dist{x}{L})=g(\|Px\|_2)$. Indeed, we can estimate the approximation error of the gradient in the following way:
 \begin{remark}\label{rem:dist} To approximate a function of the form $f(x)=g(\dist{x}{L})=g(\|Px\|_2)$, for $x\in \R^N$, $g\in C^1([0,1])$ and $g'$ Lipschitz continuous, we can utilize the following observation to obtain a worse approximation result than for linear sleeve functions of the form \eqref{eqn:linridge}. Namely, we can rewrite the $i$-th entry of the divided difference of $\|Px\|_2$ as:
 $$\nabla_h(\|Px\|_2)_i=\frac{\|P(x+he_i)\|_2-\|Px\|_2}{h}=\frac{\|P(x+he_i)\|_2^2-\|Px\|_2^2}{h\left(\|P(x+he_i)\|_2+\|Px\|_2\right)}=\frac{2(Px)_i+h\|Pe_i\|_2^2}{\|P(x+he_i)\|_2+\|Px\|_2}.$$
 We then obtain the following estimate:
\begin{align}\left| \nabla(\|Px\|_2)_i-\nabla_h(\|Px\|_2)_i\right|&=\left|\frac{2(Px)_i+h\|Pe_i\|_2^2}{\|P(x+he_i)\|_2+\|Px\|_2}-\frac{(Px)_i}{\|Px\|_2}\right|\\&\le \frac{\left|2(Px)_i\|Px\|_2-(Px)_i\|Px\|_2-(Px)_i\|P(x+he_i)\|_2\right|+h\|Pe_i\|_2^2}{\|Px\|_2\left(\|Px\|_2+\|P(x+he_i)\|_2\right)}\\&\le \frac{|(Px)_i|\|Px-P(x+he_i)\|_2}{\|Px\|_2^2}+\frac{h\|Pe_i\|_2^2}{\|Px\|^2}\le 2h\frac{\|Pe_i\|_2}{\|Px\|_2^2}.\end{align}
And, hence, we have:
$$\|\nabla(\|Px\|_2)-\nabla_h(\|Px\|_2)\|_2\le \frac{2h\sqrt{N-d}}{\|Px\|_2}.$$
If $\|Px\|_2$ is small, this upper bound can become large. Fortunately, in the case of linear sleeve functions of the form \eqref{eqn:linridge}, we are not constrained by this term (cf. Equation \eqref{eqn:aTN-1}).

\end{remark}


\section{An Optimization Algorithm on the Grassmannian Manifold}\label{sec:Opt}
We will now reformulate the given approximation problem as an optimization problem over the Grassmannian manifold $G(d,N)$. This reformulation allows us to develop an algorithm which yields a reconstruction whose error decreases with the number of sampling points. Remember that the previous algorithm needed a fixed number of sampling points and the error has decreased with the step size $h$ in the computation of the divided differences. We again use the following notation for $f$ 
\begin{align}
 f(x)=g(\dist{x}{L}^2)=g(\|Px\|_2^2),
\end{align}
where the operator $P$ denotes the orthogonal projection $P_P$ to the subspace $P\subset \R^N$ orthogonal to $L$. In the sequel of this section, we will, for brevity, assume that $\dim P=d$ where we assumed $\dim P=N-d=N-\dim L$ in the last section.
\subsection{The Algorithm}
Let us assume without loss of generality that $g$ is not the constant function. Otherwise we do not need to find the subspace, since every subspace can be used to represent $g$ as stated.
We first define for each $H\in G(d,N)$ a function $f_H$ as linear-sleeve function, namely by
\begin{equation}
 f_H(x)=g(\|Hx\|^2_2).
\end{equation}
The main idea of the algorithm then uses the fact that $f_P=f$ and $f_H\neq f$ for $H\neq P$ and, thus, that $P$ is the unique minimizer of
\begin{equation}
G(d,N)\ni H\mapsto  \int_{[0,1]^N}\left|f(x)-f_H(x)\right|^2dx.
\end{equation}
Unfortunately, we cannot \sk{express} this objective function in terms of sampling the input function $f$. On the one hand, we, therefore, need to replace the integral by a finite sum and on the other hand, the definition of $f_H$ is not clear, since we do not know $g$ in advance.

 However, note that we can easily recover $g$ by sampling $f$ in some random direction $\hat{\theta}$. Indeed, it holds for $t\in \R$ that $f(t\hat{\theta})=g(t\|P\hat{\theta}\|_2^2)$ and, since $\hat{\theta}$ is almost surely not contained in the orthogonal complement of $P$, $g$ is up to the constant $\|P\hat{\theta}\|_2^2$ uniquely determined by $f(\cdot \hat{\theta})$. Hence, if we knew $\|P\hat{\theta}\|_2^2$ approximately, sampling $f$ at  $ih\hat{\theta}$, $i=1,\dots,M$, where $h\in (0,1)$ is the step size, gave an approximation to $g$. Namely, with $g_{\hat{\theta}}=g(\cdot \|P\hat{\theta}\|_2^2)$, we let $\hat{g}_{\hat{\theta}}^M$ be the approximation of $g_{\hat{\theta}}$ from the sampling points  $(ih,f(ih\hat{\theta}))$. We can, then, set $\hat{f}_H^M(x)=\hat{g}^M_{\theta}({\hat{p}_{\hat{\theta}}}^{-1}{\|Hx\|_2^2})$, where $\hat{p}_{\hat{\theta}}$ is an approximation of $\|P\hat{\theta}\|_2^2$.

 One possibility to choose $\hat{\theta}$ such that we know $\|P\hat{\theta}\|_2^2$ approximately, is to choose $\hat{\theta}$ as the approximation of the normalized gradient of $f$ in some random direction $\eta$. Indeed, the normalized gradient of $f$ is given by
\begin{align}
  \theta=\frac{\nabla f(\eta)}{\|\nabla f(\eta)\|_2}=\frac{P\eta}{\|P\eta\|_2}.
 \end{align}
 Therefore, we have almost surely (e.g., with respect to the Lebesque measure) $\|P\theta\|_2=1$. Thus, we choose
 \begin{align}
  \hat{\theta}=\frac{\tilde{\theta}}{\|\tilde{\theta}\|_2^2}, \quad\quad \text{where} \quad\quad \tilde{\theta}_i=\frac{f(\eta+he_i)-f(\eta)}{h},\quad i=1,\dots, N,
 \end{align}
and let $\hat{g}_{\hat{\theta}}^M=Q_hg_{\hat{\theta}}$ be the approximation of $g_{\hat{\theta}}$, with $Q_h$ as introduced in the preliminaries (see Equation \eqref{Prob1}).
 

The only remaining task is now to substitute the integral by a finite sum. Hence, we aim to define the objective function
\begin{equation}\label{eqn:objWf}
\hat{F}^M: \quad G(d,N)\ni H\mapsto  \sqrt{\sum_{i=1}^n\left|f(x_i)-\hat{f}_H^M(x_i)\right|^2},
\end{equation}
for some $x_1,\dots, x_n\in \R^N$ such that $P$ is the unique minimizer of 
\begin{align}\label{eqn:objWf2}
 F: \quad G(d,N)\ni H\mapsto  \sqrt{\sum_{i=1}^n\left|f(x_i)-f_H(x_i)\right|^2},
\end{align}
to ensure that the minimizer  of $\hat{F}^M$ is a good approximation of $P$.
Certainly, $P$ can only be the unique minimizer of $F$ if it uniquely minimizes the function
\begin{equation}\label{map:Retr}
G(d,N)\ni H\mapsto \sqrt{ \sum_{i=1}^n\left|\|Px_i\|_2-\|Hx_i\|_2\right|^2}.
\end{equation}
It is, therefore, necessary to find $x_1,\dots, x_n$ such that the map
\begin{equation}\label{map:ProjRetr}
 A:  G(d,N)\ni H \mapsto \left(\|Hx_1\|_2,\dots, \|Hx_n\|_2\right)
\end{equation}
is injective. This problem is known as projection retrieval \cite{Mixon} and discussed in the next subsection.

 The proposed procedure is summarized in Algorithm \ref{Opt}, to which we refer to as \emph{OGM}.

\begin{algorithm}
For a given step size $h\in (0,1)$ such that $M:=h^{-1}\in \N$:
\begin{enumerate}
 \item Choose direction $\theta\in \mathbb{S}^{N-1}$, such that we know $\|P\theta\|_2$ approximately (see explanations above).
 \item Sample $y_i:=f(ih\theta)$, $i=1,\dots, M$.
 \item Approximate $g_{\theta}:=g(\cdot\|P\theta\|_2^2)$ by $\hat{g}_{\theta}^M=Q_h(g_{\theta})$, with knots $\left\{(ih, y_i)\right\}_{i=1}^{M}$.
 \item Approximate $\|P\theta\|_2$ by $\|P\hat{\theta}\|_2$.\label{item:aprProj}
 \item Set $\hat{f}^M_H=\hat{g}^M_{\theta}(\frac{\|H\cdot\|_2^2}{\|P\hat{\theta}\|_2^{2}})$.
 \item Minimize the objective function:
 \begin{equation}\label{apprObj}
  G(d,N)\ni H \mapsto \hat{\mathcal{F}}^{M}(H)=\sqrt{\sum_{i=1}^n\left|f(x_i)-\hat{f}_H^M(x_i)\right|^2},
 \end{equation}
 where $\hat{f}^M_H=\hat{g}^M_{\theta}(\frac{\|H\cdot\|_2^2}{\|P\hat{\theta}\|_2^{2}})$.
\end{enumerate}
\caption{OGM - Approximation by \underline{O}ptimization over \underline{G}rassmannian \underline{M}anifold\label{Opt}}
\end{algorithm}


We will now be able to prove the following main result in Subsection \ref{subsec:ProofM2}. 
\begin{theorem}\label{thm:main2} Let $f$ be a linear-sleeve function of the form \eqref{eqn:linridge}, i.e., $f\in\mathcal{LR}(s)$, $s\in (1,2]$, and $f=g(\dist{x}{L}^2)=g(\|Px\|_2^2)$ for some $d$-dimensional subspace $P\subset \R^N$.
Suppose that the derivative of $g$ is bounded from below by some positive constant, and let $\hat{P}:=\argmin_{H\in{G(N-d,N)}}\hat{\mathcal{F}}^M(H)$. Then, we have 
\begin{align*}\|\hat{P}-P\|_{\text{HS}}\lesssim M^{-s/2},\end{align*}
almost surely (e.g., with respect to Haar measure), with a constant depending only polynomial on the dimension of the space. 
In particular, if $f\in \mathcal{LR}(2)$, then 
\begin{align*}\|\hat{P}-P\|_{\text{HS}}\lesssim M^{-1}.\end{align*}
\end{theorem} 

Note that the statement holds indeed only almost surely, since we  have to ensure that $P\eta\neq 0$.

The first step is to find measurements $\{x_i\}_{i=1}^n$ which ensure that the map defined in \eqref{map:Retr} is injective. Then we can turn to the error analysis and the proof of Theorem \ref{thm:main2}.

\subsection{Projection Retrieval}\label{subsec:ProjRetr}
To find sampling points which ensure that the objective function has a unique minimizer, we consider the necessary case where $g$ is the identity and $\mathcal{F}$ is, therefore, given by $\mathcal{F}(H)=\sum_{i=1}^n\left(\|Px_i\|_2-\|Hx_i\|_2\right)^2$. Thus, $P$ is the unique minimizer of $\mathcal{F}(H)$ if and only if the sampling points $x_i$, $i=1,\dots,n$, determine $P$ uniquely, i.e., if the map
\begin{equation}
 A:  G(d,N)\ni H \mapsto \left(\|Hx_1\|_2,\dots, \|Hx_n\|_2\right)
\end{equation}
is injective. We can show the following theorem.
\begin{theorem}\label{thm:FullMeas}
For every $P\in G(d,N)$ the quantities
\begin{align}
 \|P(e_i+e_k)\|_2=:\|Px_{i,k}\|_2,\quad \quad\text{for}\quad i=1\dots,N,\quad k=1,\dots, N,
\end{align}
 uniquely determine $P$.
\end{theorem}
\begin{proof}
Let $P, H\subset \R^N$ be two subspaces and let $\{u_1,\dots,u_d\}$ be an orthonormal basis for $P$ and $\{v_1,\dots,v_d\}$ an orthonormal basis for $H$. Further, suppose that $\|Px_{i,k}\|_2^2=\|Hx_{i,k}\|_2^2$ for $i=1,\dots,N$, $k=i,\dots,N$. For $i=k$, we obtain
\begin{align*}
 \sum_{j=1}^d v_{ji}^2=\sum_{j=1}^d\langle e_i,v_j\rangle^2=\|Hx_{ii}\|_2^2=\|Px_{ii}\|_2^2=\sum_{j=1}^d u_{ji}^2.
\end{align*}
This shows that the entries of the diagonal of the projection matrices corresponding to $P$ and $H$ coincide. For the case $i\neq k$, we can compute
\begin{align*}
 \sum_{j=1}^d v_{ji}^2+2\sum_{j=1}^d v_{ji}v_{jk}+\sum_{j=1}^d v_{jk}^2=\|Hx_{ik}\|_2^2=\|Px_{ik}\|_2^2=\sum_{j=1}^d u_{ji}^2+2\sum_{j=1}^d u_{ji}u_{jk}+\sum_{j=1}^d u_{jk}^2.
\end{align*}
Therefore, using the knowledge from the case where $i$ equals $k$, this equation gives
\begin{align*}
 \sum_{j=1}^d v_{ji}v_{jk}=\sum_{j=1}^d u_{ji}u_{jk},\quad\quad \text{for} \quad i=1,\dots, N, \quad k=i+1,\dots, N
\end{align*}
Thus, due to the symmetry of a real projection matrix, both projection matrices coincide, since the left-hand side equals the $(i,k)$-th entry of the projection matrix corresponding to $H$ and the right-hand side to $P$, respectively. Therefore, we can conclude that $H=P$.
\end{proof}

We see that we need $(N^2+N)/2$ sampling points to ensure injectivity, if we choose them as suggested by the last theorem. However, we believe a smaller number of sample points should be sufficient. Indeed, we can ensure that a fewer number of sampling points are sufficient to ensure almost surely injectivity and, therefore, almost surely a unique minimizer. For this, we adapt Theorem 4 in \cite{Mixon} to the real case and deduce that to ensure almost injectivity, we can require only $(d+1)(N+d/2)$ points.
\begin{theorem}[\cite{Mixon}]\label{thm:ProjRetr}
 Draw a random subspace $P$ uniformly with respect to the Haar measure from the Grassmannian manifold $G(d,N)$. Then the quantities 
 \begin{align}
  \|Pe_i\|_2, \hspace{0.2cm} \|P(e_i+e_k)\|_2
 \end{align}
for $i\in \{1,\dots, N\}$ and $k\in \{i+1,\dots,d\}$, uniquely determine $P$ a with probability of $1$.
\end{theorem}
Note that the change in the second index set in comparison to \cite{Mixon} is due to taking the symmetry of a real projection matrix into account. And the change in the number of necessary measurements is due to some small typo in \cite{Mixon}, since for this proposed procedure we need to compute not only the first $d$ columns of the projection matrix, but also all its diagonal entries. However, also in \cite{Mixon} it is proven that the first $d$ columns of the projection matrix $P$ determine the corresponding subspace almost surely uniquely. Thus, it would be desirable to find points which allow us to directly determine the entries of the first $d$ columns, without computing all diagonal entries. This would deduce the number of necessary measurements to $Nd$. In the case $d=1$, the following result by Fickus, Mixon, Nelson, Wang \cite{Fickus},  tells us that  $N+1$ measurements are sufficient to ensure almost injectivity, which is almost the conjectured number $Nd$.
\begin{theorem}[\cite{Fickus}]\label{thm:PhaseRetr}
 Consider $\Phi = \left\{\phi_i\right\}_{i=1}^n \subset \R^N$ and the intensity measurement mapping $A: \R^N/{\pm 1} \rightarrow \R^n$ defined by $(A(x))(i) := |\langle x, \phi_i \rangle|^2$ . Suppose each $\phi_i$ is nonzero. Then $A$ is almost surely injective if and only if $\Phi$ spans $\R^N$ and $\text{rank }\Phi_S + \text{rank }\Phi_{S^C} >N$ for each nonempty proper subset $S\subset \{1,\dots,N\}$.
\end{theorem}

This shows that $A$ cannot be almost injective if $n< N+1$. Moreover, for the case $n=N+1$, it is almost injective if and only if $\Phi$ is \emph{full spark}, which means that every size-$N$ collection of vectors of $\Phi$ is linearly independent. We remark that this result does not stand in contradiction to the above mentioned conjecture that $Nd$ measurements are sufficient. Indeed, in our case we want to recover the subspace and in this sense we can interpret the condition $1=\|v\|_2=\langle v,v \rangle$, for some vector $v\in \R^N$ of a orthonormal basis of this vector, as an additional measurement. Of course, it is not generally true that every size-$N$ subcollection of $\left\{e_1,\dots, e_N,v\right\}$ forms a spanning set, because $v$ could have zero entries. However, if we consider $v$ as embedded in $\R^m$, where $\{i_1,\dots, i_m\}$ are the indices of the nonzero-entries of $v$, then every size-$m$ subcollection of $\left\{e_{i_1},\dots, e_{i_m},v\right\}$ forms a spanning set for $\R^m$. Thus by Theorem 12 in \cite{Fickus}, the entries $i_1,\dots, i_m$ of $v$ are uniquely determined by these sampling points with a probability of $1$. That the other entries are equal to zero is already determined by the other sampling points.

As indicated by this discussion, we suspect that $Nd$ sampling points are sufficient to ensure almost surely injectivity of $A$. And, indeed, we are able to prove the following theorem, which states that even a fewer number of sampling points are sufficient, although, we do not directly determine the first $d$ columns of the projection matrix.

\begin{theorem}\label{thm:ProjRetrn}
 Draw a random subspace $P$ uniformly with respect to the Haar measure from the Grassmannian manifold $G(d,N)$. Then the quantities 
 \begin{align}
  \|Pe_i\|^2_2, \hspace{0.2cm} \|P(e_j+e_k)\|^2_2, \hspace{0.2cm} \|Px\|_2^2
 \end{align}
for a randomly chosen vector $x\in \R^N$, $i\in \{1,\dots, N-1\}$, $j\in \{1,\dots,d\}$ and $k\in \{j+1,\dots,N\}$, uniquely determine $P$ with a probability of $1$.
\end{theorem}

\begin{proof}
 We start with the same argumentation as in \cite{Mixon}, which tells us that the first $d$ columns of the projection matrix $P$ are linearly independent for almost every $P$, that the diagonal entries of the projection matrix $P$ are given by
 \begin{align}
  P_{ii}=\|Pe_i\|_2^2, \quad \text{for} \quad i=1,\dots, N,
 \end{align}
and that the other entries can be computed as
\begin{align}
 2P_{ij}=\|P(e_i+e_j)\|_2^2-\|Pe_i\|_2^2-\|Pe_j\|_2^2.
\end{align}
We further note that we only need to observe $N-1$ diagonal entries to determine all $N$ diagonal entries of $P$. Indeed, if $u_1,\dots,u_d$ is any orthonormal system which spans $P$, it holds that
\begin{align}
 \sum_{i=1}^N P_{ii}=\sum_{i=1}^N \|Pe_i\|_2^2=\sum_{i=1}^N\sum_{j=1}^du_{ji}^2=d,
\end{align}
and, thus, $P_{NN}=d-\sum_{i=1}^{N-1}P_{ii}$. Hence, by observing $\|Pe_i\|_2^2$, $i=1,\dots, N-1$, as well as $\|P(e_j+e_k)\|_2^2$, $j=1,\dots, d$ and $k=j+1,\dots, N$, we can recover the first $d-1$ columns of the projection matrix $P$.

We now claim that there exist only finitely many projection matrices with the same first $d-1$ columns. This would show that there are only finitely many subspaces $H$ which yield the same measurements as $P$ for the stated collection of quantities and that we, therefore, can almost surely uniquely recover $P$ by an additional random measurement.
 
 Thus, suppose we already know the first $d-1$ columns of the projection matrix. These are clearly linearly independent in the same event, where the first $d$ columns of $P$ are linearly independent. We can now use the fact that for a projection matrix $P$, it has to hold that $PP=P$ and, hence, that each column of $P$ lies in the span of the corresponding subspace. Applying Gram-Schmidt orthonormalization to the first $d-1$ columns, which are linearly independent, gives an orthonormal system of $d-1$ vectors, which we denote by $u_1,\dots, u_{d-1}$. Thus, there is only one unknown basis vector, denoted by $u_d$, for the subspace $P$ left. However, the measurements we have already taken determine the entries of this vector uniquely in absolute value. Indeed, we have
 \begin{align}
  \|Pe_i\|_2^2=\sum_{j=1}^d u_{ji}^2=u_{di}^2+\sum_{j=1}^{d-1}u_{ji}^2,  
 \end{align}
 which is equivalent to 
 \begin{align}
   u_{di}^2=\|Pe_i\|_2^2-\sum_{j=1}^{d-1}u_{ji}^2.
 \end{align}

Note that we already know the right-hand side of the last equation. This shows that there are indeed only finitely many subspaces which produce the same measurements as $P$. Hence, taking some random measurement in addition, yields the desired almost injectivity.
%
\end{proof}


The following corollary shows that if the dimension of the subspace is $d>N/2$, we can choose the same measurements as if the dimension would be $N-d$. So we can further deduce the number of measurements. For example, if the dimension $d=N-1$, we can apply Theorem \ref{thm:PhaseRetr} and find that $N$ measurements are sufficient to ensure almost injectivity of $A$.

\begin{corollary}\label{cor:grassmin2}
With the same choice of measurements as in Theorem \ref{thm:ProjRetrn} we can uniquely determine a randomly drawn subspace $P\in G(N-d,N)$ with a probability of $1$, i.e., the measurements 
\begin{align}
  \|Pe_i\|_2^2, \hspace{0.2cm} \|P(e_i+e_k)\|_2^2, \hspace{0.2cm} \|Px\|_2^2
 \end{align}
for some random vector $x\in \R^N$, $i\in \{1,\dots, N-1\}$, $j\in \{1,\dots,d\}$ and $k\in \{j+1,\dots,N\}$, uniquely determine $P$ with probability $1$.
\end{corollary}

\begin{proof}
For every $y\in \R^N$ it holds $\|Py\|_2^2=\|Hy\|_2^2$ if and only if $\|P^\perp y\|_2^2=\|y\|_2^2-\|Py\|_2^2=\|y\|_2^2-\|Hy\|_2^2=\|H^\perp y\|_2^2$, and therefore, we can apply the results of the above theorem.
\end{proof}

\subsection{Proof of Theorem \ref{thm:main2}}\label{subsec:ProofM2}
In the  last subsection we have shown that we need less than $Nd$ sampling points to ensure that the objective function defined in \eqref{map:Retr}, where $g$ is assumed to be the identity, has almost surely a unique minimizer. However, for ease of computation we will use the \sk{sampling points proposed} in Theorem \ref{thm:FullMeas}. We first show that the measurements given in Theorem \ref{thm:FullMeas} also ensure a unique minimizer of $F_H$. For this purpose we introduce a bijective mapping $$\imath:\left\{1,\dots,N(N+1)/2\right\}\rightarrow \left\{(j,k):j\in\{1,\dots,N\}, k\in \{j,\dots,N\}\right\}$$ and set $n=N(N+1)/2$ as well as $$x_i:=x_{\imath(i)}=x_{j,k}=\begin{cases} e_j & \text{if } j=k,\\ e_j+e_k & \text{if }j\neq k,\end{cases} \hspace{0.5cm} \text{for } i=1,\dots,n.$$
\begin{lemma}
 Suppose that $g$ fulfills the assumption of Theorem \ref{thm:main2}. Then $P$ is the unique minimizer of 
 \begin{align}
  \mathcal{F}:G(d,N)\rightarrow \R, \hspace{0.2cm} H\mapsto \sqrt{\sum_{i=1}^{n} \left(f(x_i)-f_H(x_i)\right)^2},
 \end{align}
where $x_i$, $i=1,\dots,n$, are defined as above.
\end{lemma}

\begin{proof}
 Suppose $H\in G(d,N)$ is another minimizer. Then for all $i\in \left\{1,\dots, n\right\}$, we conclude $g(\|Px_i\|^2_2)=g(\|Hx_i\|^2_2)$. But since $g$ is injective, this implies
 $\|Px_i\|_2=\|Hx_i\|_2$ for all $i\in \left\{1,\dots, n\right\}$. Thus by the statements proved in subsection \ref{subsec:ProjRetr} we conclude that $P=H$.
\end{proof}

\begin{lemma}Under the same assumption as in Theorem \ref{thm:main2}, we have
 $$|\hat{\mathcal{F}}^M(H)-\mathcal{F}(H)|\le C\sqrt{d}M^{-s/2},$$ where $\hat{\mathcal{F}}^M$ was introduced in Equation \eqref{eqn:objWf}.
\end{lemma}
\begin{proof} We \sk{start by estimating $\hat{\mathcal{F}}^M(H)$} as
 \begin{align}
  \hat{\mathcal{F}}^M(H)&=\sqrt{\sum_{i=1}^n\left(f(x_i)-\hat{f}^M_H(x_i)\right)^2}\\&=\sqrt{\sum_{i=1}^n\left((f(x_i)-f_H(x_i))+(f_H(x_i)-\hat{f}^M_H(x_i))\right)^2}
  \\&\le\sqrt{\sum_{i=1}^n(f(x_i)-f_H(x_i))^2}+\sqrt{\sum_{i=1}^n(f_H(x_i)-\hat{f}^M_H(x_i))^2}
  \\&=\mathcal{F}(H)+\sqrt{\sum_{i=1}^n(f_H(x_i)-\hat{f}^M_H(x_i))^2},
 \end{align}
 where we used the triangle inequality for $\|\cdot\|_{\ell_2^n}$ in the second step.
We can apply a similar argument to $\mathcal{F}(H)$ to derive $\mathcal{F}(H)\le \hat{\mathcal{F}}^M(H)+\sqrt{\sum_{i=1}^n(f_H(x_i)-\hat{f}^M_H(x_i))^2}$. This in turn yields the inequality
\begin{align}
|\mathcal{F}(H)-\hat{\mathcal{F}}^M(H)|\le\sqrt{\sum_{i=1}^n(f_H(x_i)-\hat{f}^M_H(x_i))^2}.
\end{align}
We now split this inequality  by
\begin{align}
 |\mathcal{F}(H)-\hat{\mathcal{F}}^M(H)|&\le\sqrt{\sum_{i=1}^n\left(g(\|Hx_i\|_2^2)-g_{\hat{\theta}}(\|Hx_i\|_2^2)+g_{\hat{\theta}}(\|Hx_i\|_2^2)-\hat{g}^M_{\hat{\theta}}(\|Hx_i\|_2^2)\right)^2}\\&\le \sqrt{\sum_{i=1}^n\left(g(\|Hx_i\|_2^2)-g_{\hat{\theta}}(\|Hx_i\|_2^2)\right)^2}+\sqrt{\sum_{i=1}^n\left(g_{\hat{\theta}}(\|Hx_i\|_2^2)-\hat{g}^M_{\hat{\theta}}(\|Hx_i\|_2^2)\right)^2}\\& \le n\left(\|g-g_{\hat{\theta}}\|_{\infty}+\|g_{\hat{\theta}}-\hat{g}_{\hat{\theta}}^M\|_{\infty}\right)=:n(T_1+T_2).
\end{align}
For $T_1$, we estimate
\begin{align}
 T_1&=\sup_{t\in [0,1]}\left|g(t)-g(\|P\hat{\theta}\|_2^2t)\right|\le \|g'\|_{\infty}|1-\|P\hat{\theta}\|_2^2|=\|g'\|_{\infty}\left|\|P\frac{\nabla f(\eta)}{\|\nabla f(\eta)\|_2}\|_2^2-\|P\frac{\nabla_hf(\eta)}{\|\nabla_hf(\eta)\|_2}\|_2^2\right|\\&\le 2\|g'\|_{\infty}\|\frac{\nabla f(\eta)}{\|\nabla f(\eta)\|_2}-\frac{\nabla_hf(\eta)}{\|\nabla_hf(\eta)\|_2}\|_2\le 4\|g'\|_{\infty}\frac{\|\nabla f(\eta)-\nabla_h f(\eta)\|_2}{\|\nabla f(\eta)\|_2}\\
 &\lesssim \sqrt{d}h^{s/2},
\end{align}
with a constant depending on $\|g'\|_{\infty}$ and $\|\nabla f(\eta)\|_2$. Note that we used the estimate from Lemma \ref{claim:S0}.

%
%

Using Property \eqref{Prob1}, we can bound $T_2$ by
\begin{align}
 T_2=\sup_{t\in[0,1]}\left|g_{\hat{\theta}}(t)-g^M_{\hat{\theta}}(t)\right|\le C h^{s},
\end{align}
where $C$ is a constant depending only on the degree of the interpolating polynomials. This proves the lemma.
\end{proof}

\begin{theorem}[Convergence]
Under the assumptions of Theorem \ref{thm:main2}, suppose that $\hat{P}$ is a minimizer of $\hat{\mathcal{F}}^M$. Then 
\begin{align}
 \|P-\hat{P}\|\lesssim n(N+1)\sqrt{d}M^{-s/2},
\end{align}
with a constant depending on the H\"older norm and bounds of $g'$.
 \end{theorem}

 \begin{proof}
 Let $H_0$ be a minimizer of $\hat{\mathcal{F}}^M$. First note that
 \begin{align}
  \mathcal{F}(H_0)\le 2C\sqrt{d}h^{s/2},
 \end{align}
because $\hat{\mathcal{F}}^M(H_0)\le \hat{F}^M(P)=\hat{\mathcal{F}}^M(P)-\mathcal{F}(P)\le C\sqrt{d}h^{s/2}$, where we used the fact that $H$ is a minimizer in the first step, that $\mathcal{F}(P)=0$ in the second step and the statement of the last lemma in the third step. The stated bound then follows from $\mathcal{F}(H_0)\le \left|\mathcal{F}(H_0)-\hat{\mathcal{F}}^M(H_0)\right|+\left|\hat{\mathcal{F}}^M(H_0)\right|\le 2C\sqrt{d}h^{s/2}$. This yields

  \begin{align}\label{eqn:proofMain}
  2C\sqrt{d}M^{-s/2}\ge \mathcal{F}(H) =\sqrt{\sum_{i=1}^n\left(g(\|Px_i\|_2^2)-g(\|Hx_i\|^2_2)\right)^2}\ge \min|g'|\sqrt{\sum_{i=1}^n\left(\|Px_i\|_2^2-\|Hx_i\|_2^2\right)^2}.
 \end{align}
 
Now define the matrix $\tilde{P}$ by
  \begin{align*}
   \tilde{P}_{ii}&=\|Pe_i\|_2^2 \hspace{0.2cm}\text{ and }\hspace{0.2cm} \tilde{P}_{ij}=\|P(e_i+e_j)\|_2^2,
 \end{align*}
 and $\tilde{H}$ analogously by
 \begin{align*}
   \tilde{H}_{ii}&=\|H_0e_i\|_2^2 \hspace{0.2cm}\text{ and }\hspace{0.2cm} \tilde{H}_{ij}=\|H_0(e_i+e_j)\|_2^2.
 \end{align*}
 We further denote the matrix which only contains the diagonal entries of a matrix $P$, and is zero elsewhere, by $P_D$ and the matrix which is zero along the diagonal and coincides at all off-diagonal entries with $P$ by $P_{OD}$. We then have
 \begin{align}
  \sqrt{\sum_{i=1}^n\left(\|Px_{i}\|_2^2-\|H_0x_{i}\|_2^2\right)^2}=\|\tilde{P}_D-\tilde{H}_D+\frac{1}{\sqrt{2}}\left(\tilde{P}_{OD}-\tilde{H}_{OD}\right)\|_{\text{F}}.
 \end{align}

Note that the $\{ij\}$-th entry $P_{ij}$ of the projection matrix $P$ to the belonging subspace is given by
  \begin{align*}
   P_{ii}&=\|Pe_i\|_2^2 \hspace{0.2cm}\text{ and } \hspace{0.2cm} 2P_{ij}=\|P(e_i+e_j)\|_2^2-\|Pe_i\|_2^2-\|Pe_j\|_2^2.
  \end{align*}
 Thus, defining 
\begin{align*}
B=\begin{bmatrix}
    0&P_{11}-H_{11} &\dots& P_{11}-H_{11}\\
    P_{22}-H_{22}&0&\dots&P_{22}-H_{22}\\
    \vdots&&&\vdots\\
    P_{NN}-H_{NN}&\dots&P_{NN}-H_{NN}&0
    \end{bmatrix},
 \end{align*}   
 gives
 \begin{align*}
\|P-H\|_{\text{F}}&=\left\|\tilde{P}_D-\tilde{H}_D+\frac{1}{2}\left(\tilde{P}_{OD}-\tilde{H}_{OD}-B-B^T\right)\right\|_{\text{F}}\\&\le 
\left\|\tilde{P}_D-\tilde{H}_D+\frac{1}{\sqrt{2}}\left(\tilde{P}_{OD}-\tilde{H}_{OD}\right)\right\|_{\text{F}}+\frac{\sqrt{2}-1}{2}\|\tilde{P}_{OD}-\tilde{H}_{OD}\|_{\text{F}}+\|B\|_{\text{F}}\\&\le \sqrt{\sum_{i=1}^n\left(\|Px_{ij}\|_2^2-\|Hx_{ij}\|_2^2\right)^2} +(N-1)\|\tilde{P}_D-\tilde{H}_D\|_{\text{F}}+1/4\|\tilde(P)_{OD}-\tilde{H}_{OD}\|_{\text{F}}\\&\le(N+1)\sqrt{\sum_{i=1}^n\left(\|Px_{ij}\|_2^2-\|Hx_{ij}\|_2^2\right)^2}.
 \end{align*} 
 Applying \eqref{eqn:proofMain}, we can, therefore, deduce
 \begin{align}
  \frac{\min|g'|}{N+1}\|P-H\|_{\text{F}}\le \min|g'|\sqrt{\sum_{i=1}^n\left(\|Px_i\|_2^2-\|Hx_i\|_2^2\right)^2}\le 2C\sqrt{d}M^{-s/2},
 \end{align}
which proves the claim.
 \end{proof}

\section{Numerical Results}\label{sec:num}
In this section we investigate the performance of the approximation schemes presented in the last sections. Our algorithms separate the approximation task in approximating the one-dimensional function $g$ and the subspace $P$ independently. Consequently, the quality of the uniform approximation of $f$ by $\hat{f}$ is then bounded by the corresponding error between $g$ and $\hat{g}$ and the error between $P$ and $\hat{P}$. In what follows, we will only discuss the approximation error between $P$ and $\hat{P}$, because the approximation error between $g$ and $\hat{g}$ is well known, cf. \cite{DeVoreLorentz} and Section \ref{sec:pre}.

We consider two different functions, one which fulfills all assumptions of Theorem \ref{thm:main2}, namely $g=\tanh$, and one which does not fulfill the assumption of a positive derivative, namely $g=\sin(5\cdot)$, which is not monotone on its domain. We further consider for the dimension of the subspace $d=1$ and $8$. For the dimension of the ambient space we choose $N=10$ and $50$. For each combination of $N$, $d$ and $g$ we ran $100$ experiments, where we drew a subspace $P\in G(d,N)$ uniformly at random. Note that the analysis of the algorithm makes heavy use of the monotonicity of $g$. However, the numerics show that we have comparably good results for the non-monotonic function $\sin(5\cdot)$.

\subsection{Numerical Results for ATPE}
The implementation of ATPE (Algorithm \ref{GE}) is straightforward. However, to draw a random vector from some tangent plane, we used a method of the Matlab toolbox Manopt \cite{Boumal} to draw a random vector of $\R^N$ and projected it to the tangent plane. The results can be seen in Figure \ref{fig:gradest}. They show as promised that the error of the approximation becomes arbitrarily small for all considered choices of $d,N$ and $g$, if we choose $h$ in computing the divided differences certainly small.

\begin{figure}[H]
 \includegraphics[scale=0.42]{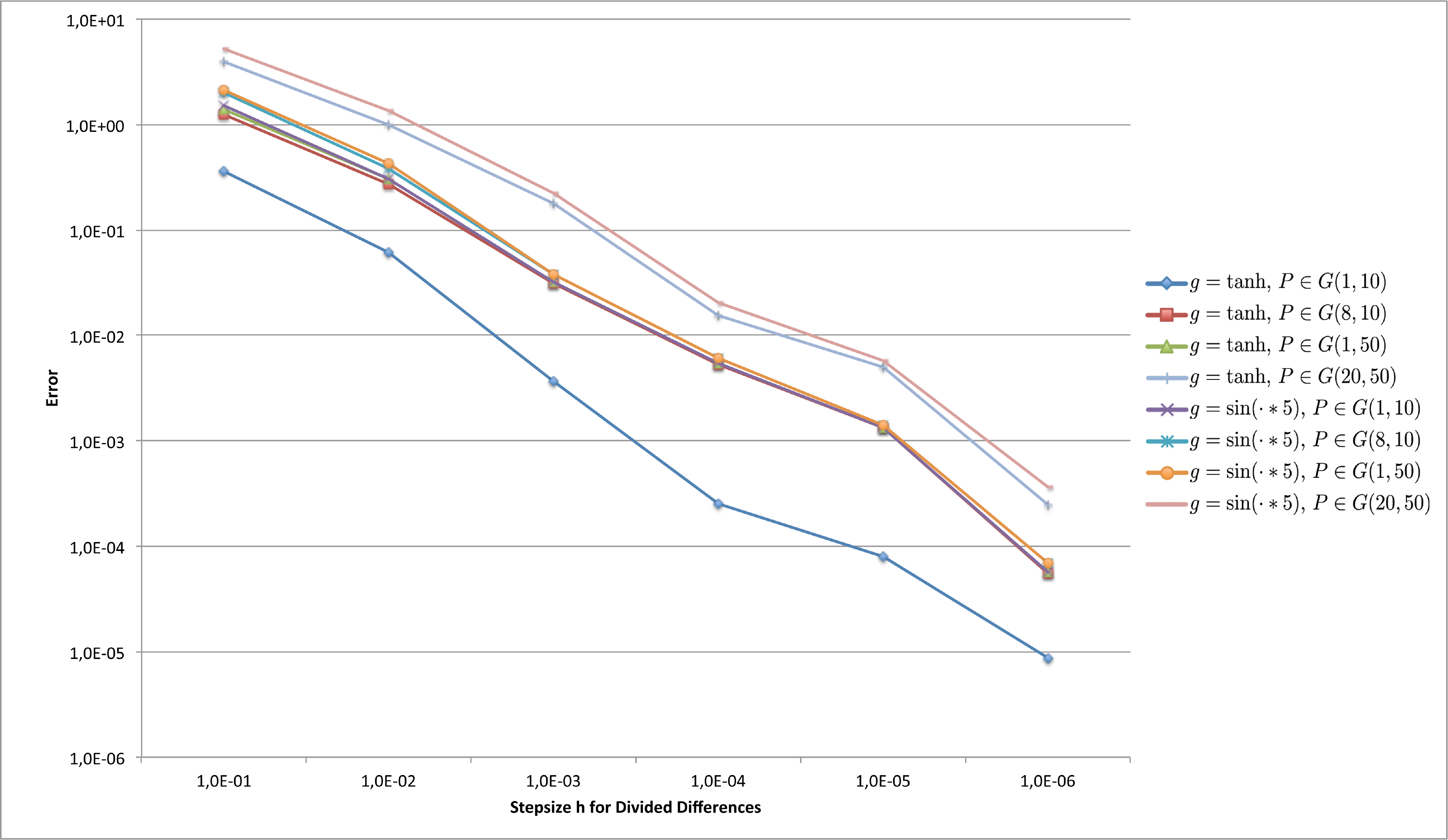}
 \caption{Average error of the approximation of randomly drawn subspaces $P$ using ATPE depending of the step size $h$ in computing $\nabla_h f$. }\label{fig:gradest}
\end{figure}

\subsection{Numerical Results for OGM (Algorithm \ref{Opt})}

To solve the optimization problem \eqref{apprObj}, we leverage the freely available Matlab toolbox Manopt \cite{Boumal}. In particular, we imposed the manifold constraint by choosing the built-in \texttt{grassmannfactory} and we selected the built-in \texttt{steepestdescent} solver, as steepest descent is a well-known method to solve optimization problems. This solver requires both a cost function and the Euclidean gradient of the cost function as inputs:\small{
\begin{align}
 \texttt{cost(H)}&= 0.25\sum_{k=1}^n\left(f(x_k)-\texttt{interp1}(\{ih\}_{i=1}^N, \{f(\|P(ih\eta)\|_2^2)\}_{i=1}^M, \|Hx_k\|_2^2,\text{'spline'})\right)^2=: 0.25\sum_{k=1}^n\left(f(x_k)-\text{\textbf{f}}(x_k)\right)\\
 \texttt{egrad}&=\sum_{k=1}^n\left(f(x_k)-\text{\textbf{f}}(x_k)\right)(\texttt{interp1}(\{ih\}_{i=1}^N, \{f(\|P(ih\eta)\|_2^2)\}_{i=1}^M, \|Hx_k\|_2^2+h/100,\text{'spline'})-\text{f}(x_k))\text{\textbf{H}}_k
,
\end{align}}
\normalsize
where $\text{\textbf{H}}_k=\left[H_1x_kx_k \dots H_Nx_kx_k\right]$ and $\texttt{interp1}(x,v,xq,\text{'spline'})$ returns interpolated values of a one-dimen\-sional function at specific query points $xq$ using spline interpolation. The vector $x$ contains the sample points, and $v$ contains the corresponding function values. Note, that the Euclidean gradient ignores the manifold constraints. 

Further observe, that even so OGM is shown to succeed to find an objective function whose minimizer is a suitable approximation of the wanted subspace, it is not obvious that the optimization algorithm can succeed to find this minimizer. In order to hope for this, we need to input a default subspace to the optimization algorithm which is not to far from the wanted one. In the following we choose those default subspaces uniformly at random in two different neighborhoods of the wanted subspace.  

Figure \ref{fig:num2} shows the results for both functions and the different choices of $N$ and $d$, if the default value for the optimization is a randomly chosen subspace in a distance of at most $\sqrt{2d(1-\cos(\pi/3))}$ to the unique minimum $P$ of the objective function \eqref{eqn:objWf}, i.e., if the default value is an rotation of $P$ with an angle of at most $\pi/3$. The error is given in a logarithmic scale and the lines correspond to the different choices of $g$, $N$ and $d$. We see that we can recover all randomly drawn subspaces successfully, whenever the dimension is $d=1$ or whenever $g=\tanh$. This fits to our analysis, where the theorems hold true for injective functions and indeed $\sin(5\cdot)$ is not at all injective.

Figure \ref{fig:num3} shows that for the case that the dimension is $d=8$ and and that we have $g=\sin(5\cdot)$, we can recover $95\%$ of randomly drawn subspaces if we ensure that the default value for the optimization is at most in a distance of $\sqrt{2d(1-\cos(\pi/4))}$ to $P$. Thus, we see that for a more carefully chosen default value, all subspaces can be recovered with a reasonable small error.
Note that the case $d=1$ corresponds to the case of a usual ridge function, because if $\dim P=1$ we measure the distance to the $N-1$-dimensional subspace $P^\perp$.

\begin{figure}[H]\begin{center}
                   \subfigure[Average]{\includegraphics[scale=0.235]{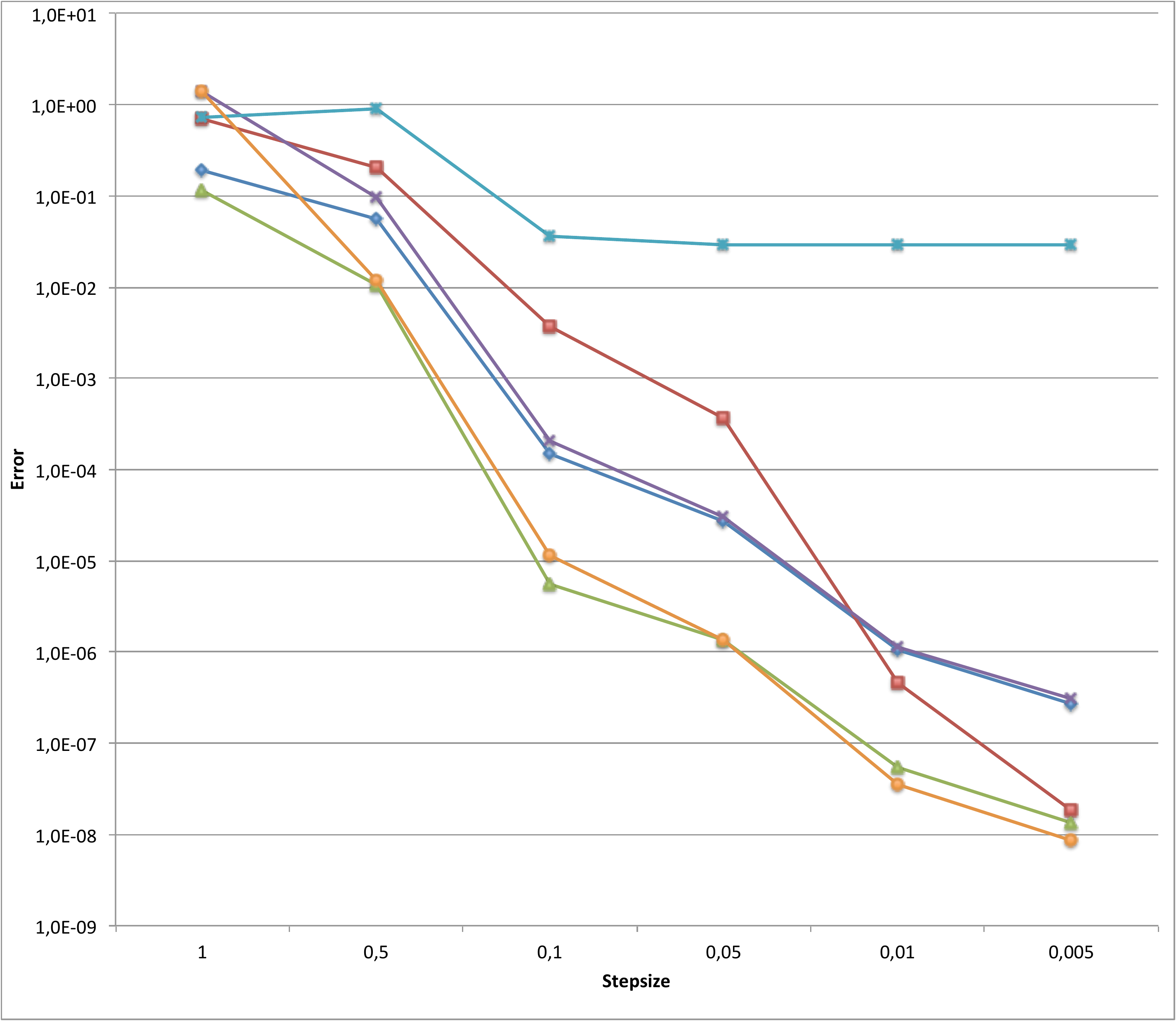}}\hspace*{0.2cm}
		   \subfigure[$95$\% Quantile]{\includegraphics[scale=0.235]{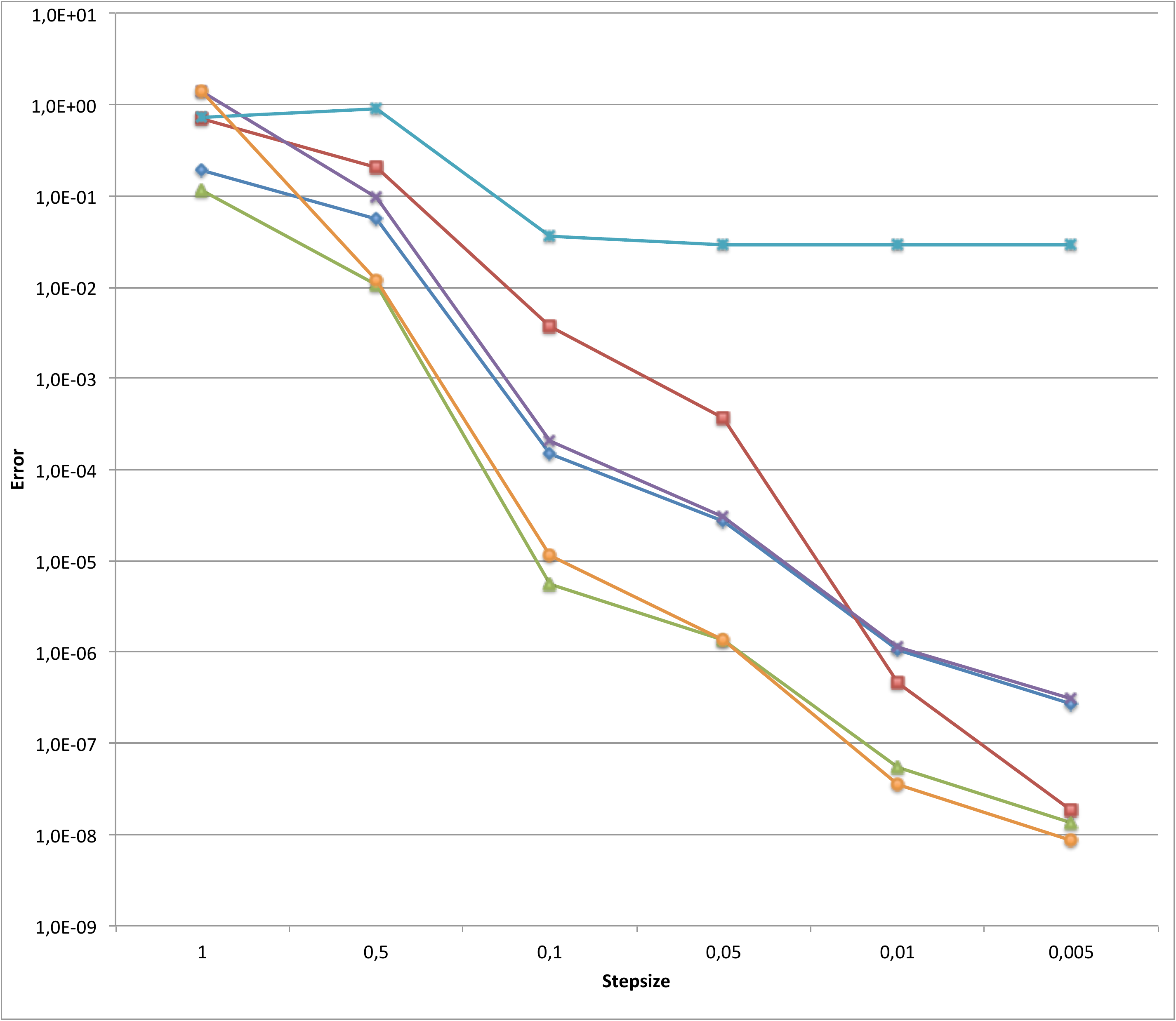}}
		   \subfigure{\includegraphics[scale=0.35]{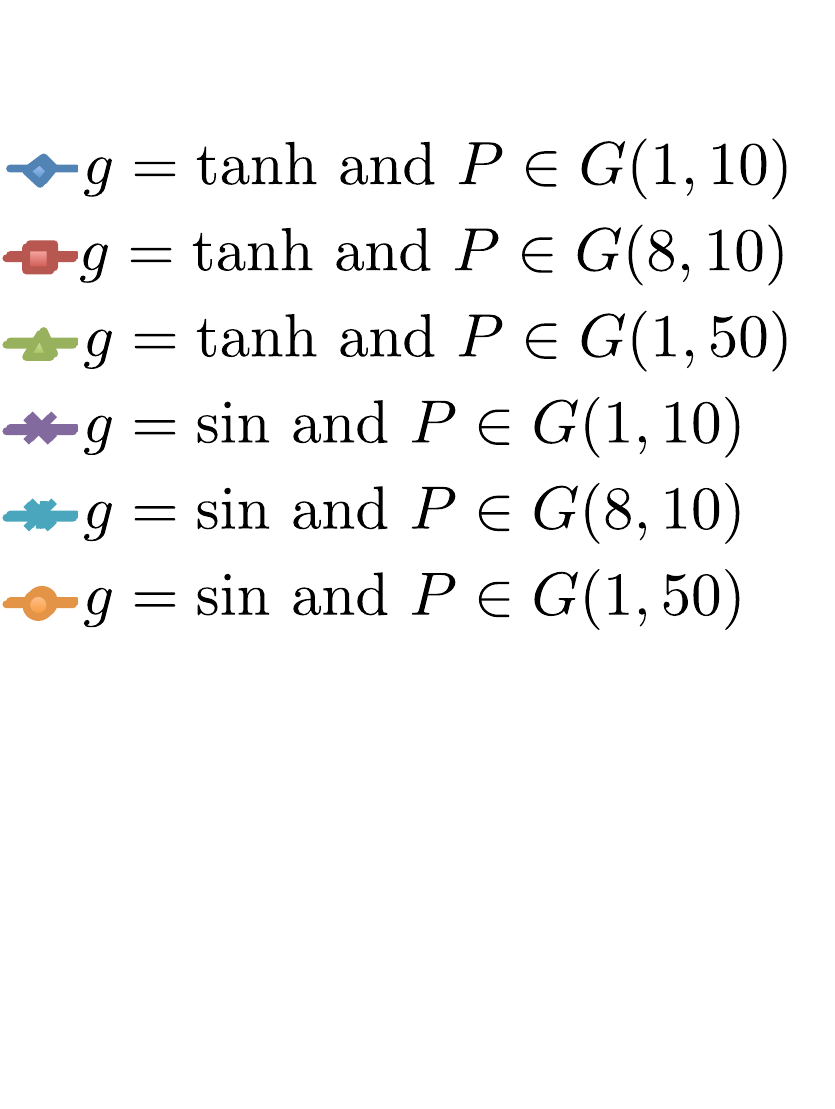}}\end{center}
 \vspace{-0.6cm}
 \caption{Default value for optimization is a random rotation of $P$ by a factor of at most $\pi/3$.}\label{fig:num2}\end{figure}
\begin{figure}[H]\begin{center}
                   \subfigure[Average]{\includegraphics[scale=0.235]{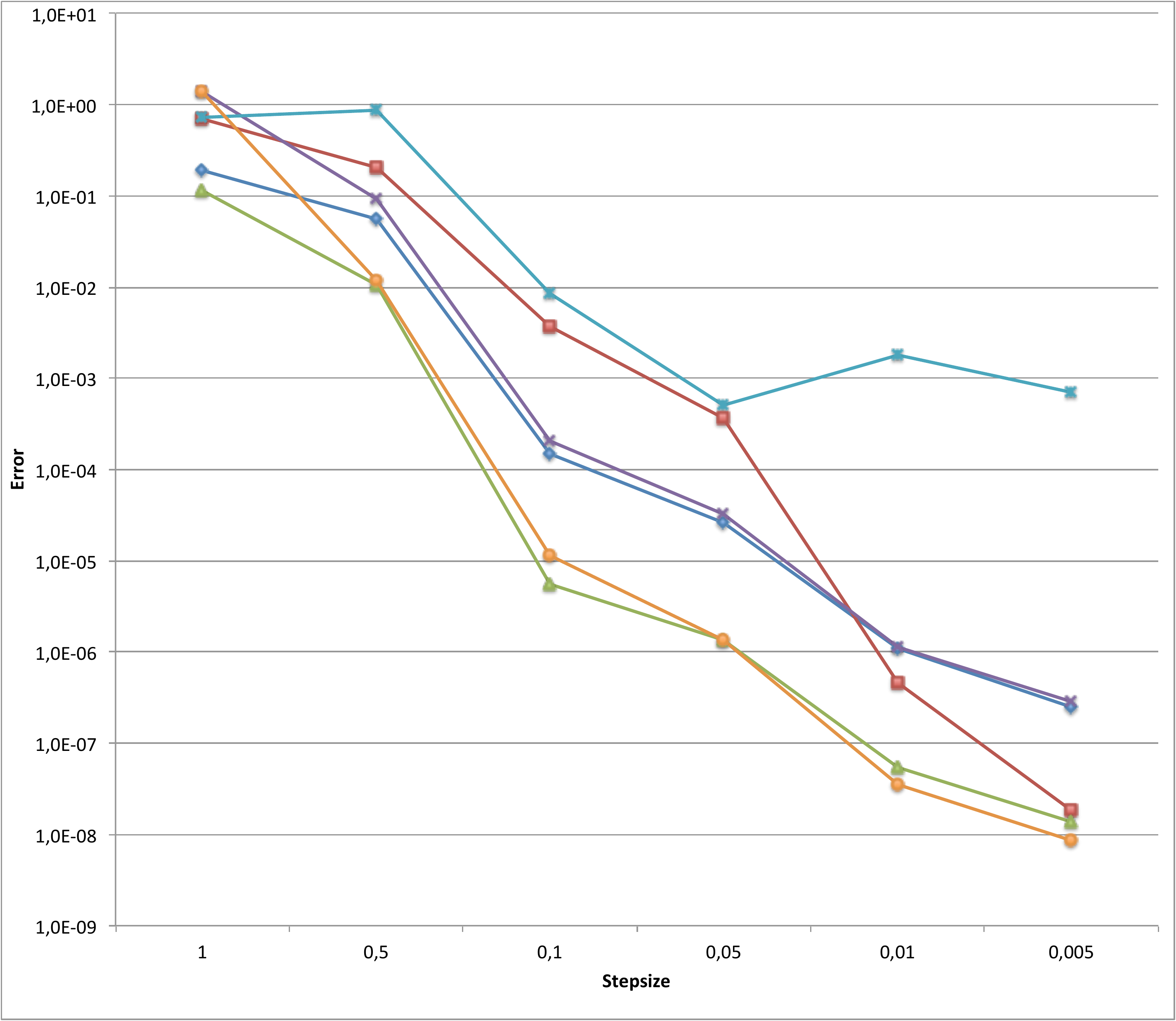}}\hspace*{0.2cm}
		   \subfigure[$95$\% Quantile]{\includegraphics[scale=0.235]{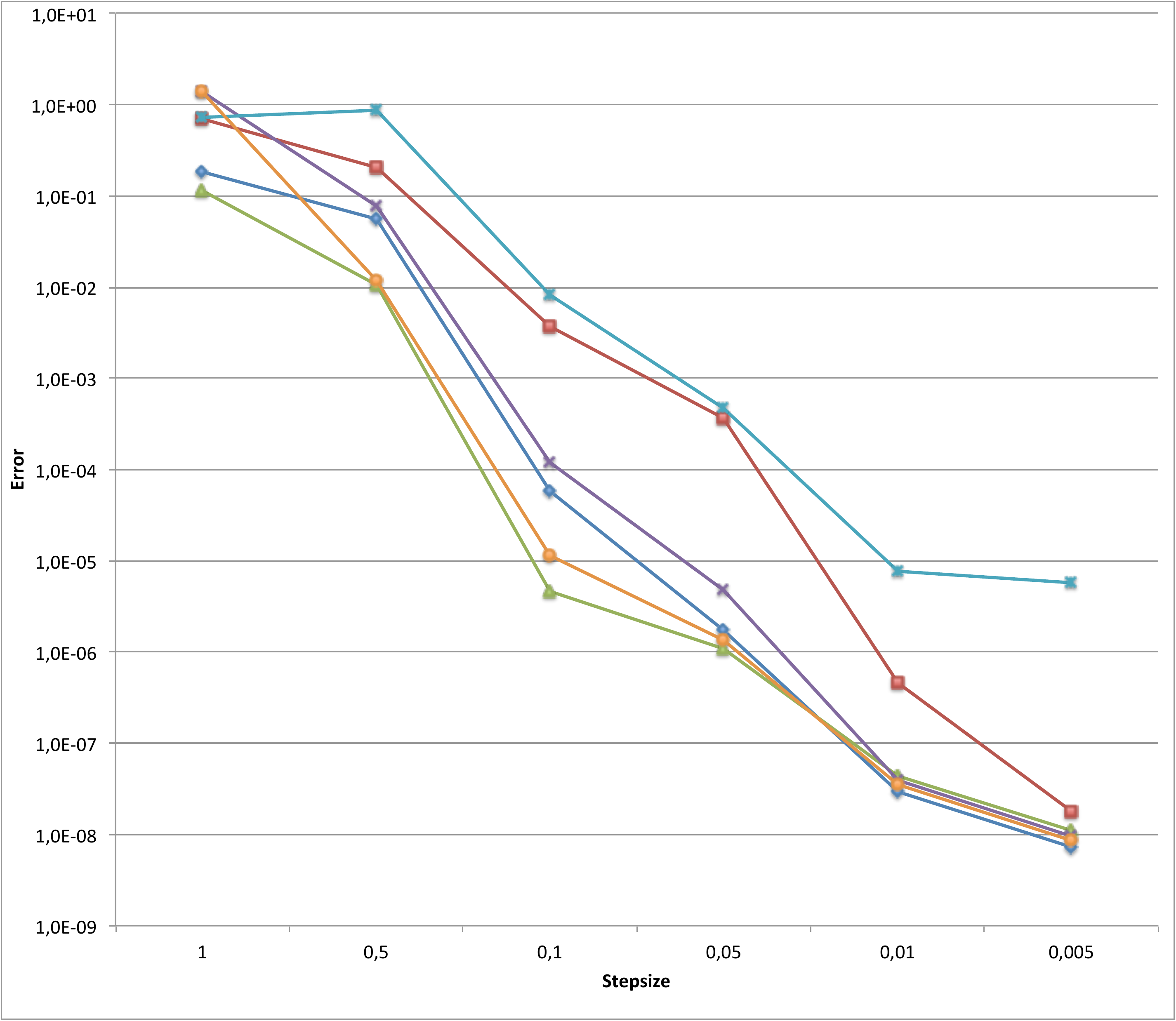}}
		   \subfigure{\includegraphics[scale=0.35]{ManOpt/Ergebnis/legend2.pdf}}\end{center}
 \vspace{-0.6cm}
 \caption{Default value for optimization is a random rotation of $P$ by a factor of at most $\pi/4$.}\label{fig:num3}\end{figure}

Furthermore, as we have seen in Subsection \ref{subsec:ProjRetr} we can use fewer measurements to ensure almost injectivity. We then also have to adapt the convegrence analysis of $\hat{F}_M$.

%
%
%
%